\documentclass[preprint,11pt]{elsarticle}
\makeatletter
\def\ps@pprintTitle{%
 \let\@oddhead\@empty
 \let\@evenhead\@empty
 \def\@oddfoot{\centerline{\thepage}}%
 \let\@evenfoot\@oddfoot}
\makeatother
\usepackage{amssymb,amsmath,amsthm}
\usepackage{epsfig}
\usepackage{color}
\usepackage{makeidx}
\usepackage{bbm}
\usepackage[top=1in, bottom=1in, inner=1in, outer=1in]{geometry}
\usepackage{setspace}
\usepackage{enumitem}

\newtheorem{theorem}{Theorem}

\newtheorem*{thmc}{Theorem C}
\newtheorem*{propa}{Proposition A}
\newtheorem*{propb}{Proposition B}

\newtheorem{definition}{Definition}
\newtheorem{lemma}{Lemma}

\newtheorem{proposition}{Proposition}

\newtheorem{remark}{Remark}

\newcommand*\xbar[1]{%
  \hbox{%
    \vbox{%
      \hrule height 0.5pt 
      \kern0.4ex
      \hbox{%
        \kern-0.15em
        \ensuremath{#1}%
        \kern-0.15em
      }%
    }%
  }%
}

\begin{document}

\begin{frontmatter}

\title{Branching Brownian motion under soft killing}

\author{Mehmet \"{O}z}
\ead{mehmet.oz@ozyegin.edu.tr}
\ead[url]{https://faculty.ozyegin.edu.tr/mehmetoz/}

\address{Department of Natural and Mathematical Sciences, Faculty of Engineering, \"{O}zye\u{g}in University, Istanbul, T\"urkiye}

\begin{abstract}
We study a $d$-dimensional branching Brownian motion (BBM) among Poissonian obstacles, where a random \emph{trap field} in $\mathbb{R}^d$ is created via a Poisson point process. In the soft obstacle model, the trap field consists of a positive potential which is formed as a sum of a compactly supported bounded function translated at the atoms of the Poisson point process. Particles branch at the normal rate outside the trap field; and when inside the trap field, on top of complete suppression of branching, particles are killed at a rate given by the value of the potential. Under soft killing, the probability that the entire BBM goes extinct due to killing is positive in almost every environment. Conditional on ultimate survival of the process, we prove a law of large numbers for the total mass of BBM among soft Poissonian obstacles. Our result is quenched, that is, it holds in almost every environment with respect to the Poisson point process.     
\end{abstract}

\vspace{3mm}

\begin{keyword}
Branching Brownian motion  \sep law of large numbers \sep Poissonian traps \sep random environment \sep soft killing
\vspace{3mm}
\MSC[2020] 60J80 \sep 60K37 \sep 60F05 \sep 60J65
\end{keyword}

\end{frontmatter}

\pagestyle{myheadings}
\markright{BBM under soft killing\hfill}

\section{Introduction}\label{intro}

In this work, we consider a model of a spatial random process in a random environment in $\mathbb{R}^d$, where the random process is a $d$-dimensional branching Brownian motion (BBM), and the random environment is created via a Poisson point process (PPP). We will call an environment in $\mathbb{R}^d$ \emph{Poissonian} if its randomness is created via a PPP. A random \emph{trap field} is formed as a positive potential which is given by the sum of a compactly supported, positive, bounded function translated at the atoms of the PPP. We specify the interaction between the BBM and the random trap field via the \emph{soft killing} rule: the particles are killed at a rate given by the value of the potential inside the trap field. Furthermore, the branching of particles is assumed to be completely suppressed inside the trap field whereas particles branch at a fixed rate outside the trap field. We call the model described here, the model of \emph{BBM among soft obstacles}. We study the growth of mass, that is, the population size, of a BBM evolving in a typical such random environment in $\mathbb{R}^d$. Clearly, the presence of traps tends to decrease the mass compared to a `free' BBM, that is, a BBM in $\mathbb{R}^d$ without any traps. The goal of this paper is to prove a law of large numbers on the reduced mass of the BBM among soft obstacles.  

The mass of BBM among random obstacles in $\mathbb{R}^d$ was first studied by Engl\"ander \cite{E2008}, where the random environment was composed of spherical traps of fixed radius with centers given by a PPP, and the interaction between the BBM and the trap field was given by the \emph{mild} obstacle rule: when a particle is inside the traps, it branches at a positive rate lower than usual and there is no killing of particles. Engl\"ander showed that on a set of full measure with respect to the PPP, a kind of law of large numbers holds (see [6, Theorem 1]) for the mass of the process. His result was later improved by \"Oz \cite{O2021} to a strong law of large numbers, including the case of zero branching in the trap field. In both aforementioned works, the challenging part of the proof was the lower bound of the law of large numbers. Under soft killing considered here, the proof of the lower bound is even more delicate as the system tends to produce fewer particles due to possible killing compared to the case of mild obstacles, and there is positive probability for the entire process to be killed by the trap field in finite time. Therefore, one has to condition the process on survival for meaningful results. It is more challenging to show that sufficiently many particles are produced with high probability under soft killing, because at each step of the proof one has to overcome the effect of possible killing of particles, which makes the analysis significantly more elaborate compared to the case of mild obstacles with zero branching inside the trap field.

\subsection{The model}

We now present the model in more detail. Firstly, we introduce the two sources of randomness, the BBM and the random trap field, and then we develop a model of a random process in random environment by specifying an interaction between the random components. 

\textbf{1. Branching Browian motion:} Let $Z=(Z_t)_{t\geq 0}$ be a strictly dyadic $d$-dimensional BBM with branching rate $\beta>0$, where $t$ represents time. The process can be described as follows. It starts with a single particle, which performs a Brownian motion in $\mathbb{R}^d$ for a random lifetime, at the end of which it dies and simultaneously gives birth to two offspring. Starting from the position where their parent dies, each offspring particle repeats the same procedure as their parent independently of others and the parent, and the process evolves through time in this way. All particle lifetimes are exponentially distributed with parameter $\beta>0$. For each $t\geq 0$, $Z_t$ can be viewed as a finite discrete measure on $\mathbb{R}^d$, which is supported at the positions of the particles at time $t$. We use $P_x$ and $E_x$, respectively, to denote the law and corresponding expectation of a BBM starting with a single particle at $x\in\mathbb{R}^d$. For $t\geq 0$ and a Borel set $A\subseteq\mathbb{R}^d$, we write $Z_t(A)$ to denote the mass of $Z$ falling inside $A$ at time $t$, and set $N_t=|Z_t|=Z_t(\mathbb{R}^d)$ to be the (total) mass of BBM at time $t$.

\textbf{2. Trap field:} The random environment in $\mathbb{R}^d$ is created as follows. Let $\Pi$ be a Poisson point process in $\mathbb{R}^d$ with constant intensity $\nu>0$, and $(\Omega,\mathbb{P})$ be the corresponding probability space with expectation $\mathbb{E}$. We now describe a way to obtain a random trap field out of $\Pi$, along with the corresponding trapping rule, which serves as the interaction between the BBM and the Poissonian trap field.


\textbf{Soft obstacles:} Consider a positive, bounded, measurable, and compactly supported \emph{killing function} $W:\mathbb{R}^d\to (0,\infty)$, and for $\omega=\sum_i \delta_{x_i}\in\Omega$ and $x\in\mathbb{R}^d$, define the potential
\begin{equation} 
V(x,\omega)=\sum_i W(x-x_i). \label{eqpotential}
\end{equation}
In this case, the Poissonian trap field $K=K(\omega)$ in $\mathbb{R}^d$ is formed as follows: 
\begin{equation} \label{eqtraprule}
 x\in K(\omega) \:\:\Leftrightarrow\:\:   V(x,\omega)>0  . 
\end{equation}
The \emph{soft killing} rule is that particles branch at the normal rate $\beta$ when outside $K$, whereas inside $K$ they are killed at rate $V=V(x,\omega)$ and their branching is completely suppressed. Note that the special case of constant killing rate inside spherical traps defined in \eqref{eqtrapfield} below corresponds to taking $W=\alpha\mathbbm{1}_{\bar{B}(0,a)}$ with $\alpha>0$ except that $W$ is not summed on overlapping balls. A formal treatment of BBM killed at rate $V=V(x,\omega)$ in $\mathbb{R}^d$ is given in \cite{LV2012}.

For $\omega\in\Omega$ we refer to $\mathbb{R}^d$ with $K(\omega)$ attached simply as the random environment $\omega$, and use $P^\omega_x$ to denote the conditional law of the BBM in the random environment $\omega$. For simplicity, set $P^\omega=P^\omega_0$. Observe that under the law $P^\omega$ the BBM has a spatially dependent branching rate
$$  \beta(x,\omega):=\beta\,\mathbbm{1}_{K^c(\omega)}(x).  $$
The main objective of this paper is to prove a quenched law of large numbers (LLN) for the mass of BBM among the Poissonian trap field introduced above. 

We now briefly describe the mild obstacle problem for BBM, which was studied in \cite{E2008} and \cite{O2021}, and serves as motivation to study the current problem. Let the trap field be given by the random set
\begin{equation}
K=K(\omega):=\bigcup_{x_i\in\,\text{supp}(\Pi)}\bar{B}(x_i,a), \label{eqtrapfield}
\end{equation}
where $\bar{B}(x,a)$ denotes the closed ball of radius $a$ centered at $x\in\mathbb{R}^d$. The \emph{mild obstacle} rule is as follows: when a particle of BBM is outside $K$, it branches at rate $\beta_2>0$, whereas when inside $K$, it branches at a lower rate $\beta_1$ with $0\leq\beta_1<\beta_2$. That is, under the law $P^\omega$, the BBM has a spatially dependent branching rate
$$  \beta(x,\omega):=\beta_2\,\mathbbm{1}_{K^c(\omega)}(x)+\beta_1\,\mathbbm{1}_{K(\omega)}(x).$$
We note that $\beta_1$ was taken to be strictly positive in \cite{E2008}, whereas in \cite{O2021} the case of $\beta_1=0$ was allowed. There is no killing of particles in the mild obstacle model. 

Unlike the mild obstacle setting, under soft killing one can show that on a set of full $\mathbb{P}$-measure there is positive probability for the entire process to be killed in finite time (see Proposition~\ref{prop1}). Therefore, to obtain meaningful results, the process is conditioned on the event of ultimate survival. Recall that $N_t=|Z_t|$ denotes the mass of BBM at time $t$. Let 
\begin{equation} \label{survival}
S_t=\{N_t\geq 1\},   \quad \quad S = \bigcap_{t\geq 0} S_t
\end{equation}
be, respectively, the event of survival up to time $t$, and the event of ultimate survival. We may also write $S_t=\{\tau>t\}$, where $\tau=\inf\{s\geq 0:N_s=0\}$. By continuity of measure from above, one deduces that $\lim_{t\to\infty} P^\omega(S_t) = P^\omega(S)$. Define the law  $\widehat{P}^\omega$ as 
$$ \widehat{P}^\omega(\:\cdot\:):=P^\omega(\:\cdot\:\mid S). $$  
Compared to the mild obstacle problem, the main extra challenge is to show that even under soft killing, in almost every environment conditional on $S$ exponentially many particles are produced for large times with overwhelming probability. This is carried out in Part 1 of the proof of the lower bound of Theorem~\ref{thm1} by making critical use of Lemma~\ref{lemma1} and Lemma~\ref{lemma2}.

For quick reference, the following list collects the different probabilities that we use in this paper. The corresponding expectations will be denoted by similar fonts. In what follows, \emph{free} refers to the model where there is no trap field in $\mathbb{R}^d$.

(i) $\mathbb{P}$ is the law of a homogeneous Poisson point process,

(ii) $P_x$ are the laws of free BBM started by a single particle at $x\in\mathbb{R}^d$,

(iii) $P_x^\omega$ are the conditional laws of BBM started by a single particle at $x$ in the environment $\omega$,

(iv) $\widehat{P}_x^\omega(\:\cdot\:):=P_x^\omega(\:\cdot\: \mid S)$, where $S$ is the event of ultimate survival of the BBM from killing,

(v) $\mathbf{P}_x$ are the laws of free Brownian motion started at $x\in\mathbb{R}^d$,

(vi) $\mathbf{P}_x^\omega$ are the conditional laws of Brownian motion started at $x$ in the environment $\omega$.

\subsection{History}

The aim of this section is to lay the background literature for the current work and to put it in perspective. The study of spatial branching processes among random obstacles in $\mathbb{R}^d$ has originated from Engl\"ander \cite{E2000}, in which a BBM among hard Poissonian obstacles was investigated with the killing rule of \emph{trapping of the first particle}. That is, the entire process is killed at the first hitting of the BBM to the trap field $K$ as opposed to killing only the particle that hits $K$. Equivalently, the event of survival up to time $t$ is defined as 
\begin{equation}
\{T>t\} \quad \text{where} \quad T=\inf\{s\geq 0: Z_t(K)\geq 1 \} .  \label{firstparticle} 
\end{equation}
In \cite{E2000}, Engl\"ander considered a uniform field of traps, and obtained the large-time asymptotic behavior of the annealed probability of survival in $d\geq 2$. Then, Engl\"ander and den Hollander studied in \cite{EH2003} the more interesting case where the trap intensity was radially decaying as
\begin{equation} \label{radialdecay}
\frac{\text{d}\nu}{\text{d}x} \sim \frac{\ell}{|x|^{d-1}}, \quad |x|\rightarrow\infty, \quad \ell>0 , 
\end{equation}
where $\text{d}\nu/\text{d}x$ denotes the density of the mean measure of the PPP with respect to the Lebesgue measure. It was shown in \cite{EH2003} that the decay rate in \eqref{radialdecay} is interesting, because it gives rise to a phase transition at a critical intensity $\ell=\ell_{cr}>0$, at which the behavior of the system changes both in terms of the large-time asymptotics of the annealed survival probability and in terms of the optimal survival strategy. In both \cite{E2000} and \cite{EH2003}, the branching rule was taken as strictly dyadic. Then, in \cite{OCE2017}, the asymptotic results for the survival probability of the system studied in \cite{EH2003} were extended to the case of a BBM with a generic branching law, including the case $p_0>0$, where $p_0$ is the probability that a particle gives no offspring at the end of its lifetime, so that a second mechanism of extinction for the BBM is intrinsically present other than that of the traps. Recently in \cite{OE2019}, conditioning the BBM on the event of survival from hard Poissonian obstacles, \"Oz and Engl\"ander proved several optimal survival strategies in the annealed setting, with particular emphasis on the population size. All works mentioned thus far assumed the hard killing rule in \eqref{firstparticle}.  

In \cite{E2008}, Engl\"ander proposed the mild obstacle problem for BBM, that is, there is no killing of particles but the branching is decreased to a nonzero constant inside $K$, and showed that on a set of full measure with respect to the PPP a kind of LLN holds for the mass of the process. This quenched result was recently improved in \cite{O2021} to the strong law of large numbers, allowing the possibility of no branching inside $K$. The current work is mainly motivated by \cite{E2008} and \cite{O2021}, and aims at proving an LLN for the mass of BBM under soft killing in $\mathbb{R}^d$. We also note that a related problem where a critical BBM that is killed at a small rate $\varepsilon>0$ inside soft obstacles was studied in \cite{LV2012}, where the main problem was to find the asymptotics of the probability that the BBM ever goes outside the ball of radius $R$ centered at the origin if $R$ is large. 

We refer the reader to \cite{E2007} for a survey, and to \cite{E2014} for a detailed treatment on the topic of BBM among random obstacles. Also, we note that the problem of LLN for spatial branching processes in a free environment in $\mathbb{R}^d$, that is, without obstacles, dates back to \cite{W1967}, where an almost sure result on the asymptotic behavior of certain branching Markov processes was established, covering the SLLN for local mass of BBM in fixed Borel sets in $\mathbb{R}^d$ as a special case. For more on the LLN in a free environment, one can see \cite{B1992} and \cite{EHK2010}, where the former work proves SLLN for spatial branching processes in linearly moving Borel sets in both the discrete setting of branching random walk in discrete time and the continuous setting of BBM, and the latter studies the local growth of mass for a large class of branching diffusions. Also, Chapters $2-4$ of \cite{E2014} contains a thorough exposition about the SLLN of branching diffusions in various settings.


\subsection{Outline} 

The rest of the paper is organized as follows. In Section~\ref{section2}, we present our main result. Section~\ref{section3} contains several preparatory results for the proof of Theorem~\ref{thm1}. In Section~\ref{section4}, we construct the almost sure (a.s.) environment that will be used in the soft obstacle problem. In Section~\ref{section5}, we state and prove a key lemma that will serve as a first step for the proof of our main result. In Section~\ref{section6}, we present the proof of Theorem~\ref{thm1}. Section~\ref{section8} discusses some further related problems.

\section{Main Result} \label{section2}

In this section, we present our main result. To this end, we introduce further notation, and define two relevant constants. Let $\omega_d$ denote the volume of the $d$-dimensional unit ball, and $\lambda_{d,r}$ denote the principal Dirichlet eigenvalue of $-\frac{1}{2}\Delta$ on $B(0,r)$ in $d$ dimensions. Set $\lambda_d=\lambda_{d,1}$. Recall that $\nu>0$ is the constant intensity of the PPP, and define the positive constants
\begin{equation} \label{eqro}
R_0=R_0(d,\nu):=\left(\frac{d}{\nu \omega_d}\right)^{1/d}
\end{equation}
and
\begin{equation} \nonumber
c(d,\nu):=\lambda_d \left(\frac{d}{\nu \omega_d}\right)^{-2/d}.
\end{equation}   
With these definitions, observe that $R_0=R_0(d,\nu)=\sqrt{\lambda_d/c(d,\nu)}$. Also, recall the law $\widehat{P}^\omega(\:\cdot\:)=P^\omega(\:\cdot\:\mid S)$ with $S$ as in \eqref{survival}.

\begin{theorem}[Quenched LLN for BBM among soft Poissonian obstacles, $d\geq 2$]\label{thm1}
Let the random environment in $\mathbb{R}^d$ be given by \eqref{eqpotential} and \eqref{eqtraprule}. Then, under the soft killing rule, in $d\geq 2$, on a set of full $\mathbb{P}$-measure,
\begin{equation} \label{eqthm1}
\underset{t\rightarrow\infty}{\lim} (\log t)^{2/d}\left(\frac{\log N_t}{t}-\beta\right)=-c(d,\nu) \quad \text{in}\:\: \widehat{P}^\omega\text{-probability}  .
\end{equation}
\end{theorem}

\begin{remark}[Quenched LLN]
Theorem~\ref{thm1} is called  a law of large numbers, because it says that the mass of BBM among Poissonian obstacles grows as its expectation (see Proposition~\ref{prop2}) as $t\rightarrow\infty$ in the sense of convergence in probability. The reason why it is called quenched is that it holds on a set of full $\mathbb{P}$-measure, that is, in almost every environment. 
\end{remark}

\begin{remark}[Robustness]
Observe that the result \eqref{eqthm1} does not depend on the details of the killing potential $V$, such as $\sup_{x}W(x)=\sup_{x}V(x,\omega)$ or $\sup_{x\in K_0}\{|x|:W(x)>0\}$, where $K_0$ stands for the compact on which $W$ is supported. Moreover, in \cite[Theorem 1]{O2021}, the same formula as in \eqref{eqthm1} (except that there was no conditioning on ultimate survival) was obtained as the SLLN for BBM among mild obstacles, where the branching was totally or partially suppressed inside the traps but there was no killing of particles. This means, the result is not only unaffected by the fine details of the trapping mechanism, but is also unaffected by the nature of the traps, whether they be soft or mild. Therefore, Theorem~\ref{thm1} above and \cite[Theorem 1]{O2021} suggest that the LLN for the mass of BBM among Poissonian obstacles is quite robust to the nature and details of the trapping mechanism. 

This can be explained as follows. In almost every environment, the mass of BBM to the leading order is entirely determined by what happens inside large trap-free regions rather than what happens inside the traps. In more detail, `large' clearings (see Definition~\ref{def1}) are present in almost every environment in the case of mild traps but also even under soft killing via a potential of the form in \eqref{eqpotential} (to make a connection between the two models, set $a=\sup_{x\in K_0}\{|x|:W(x)>0\}$, where $a$ is the trap radius in \eqref{eqtrapfield} in the case of mild traps), and the BBM is able to hit these clearings soon enough with overwhelming probability regardless of the details of the trapping mechanism. Once the BBM hits such a large clearing, the sub-BBM emanating from the particle that hits the clearing is able to produce sufficiently many particles within this clearing in the remaining time. It is the growth inside the large clearing that determines the mass, to the leading order, of the BBM for large times. Obviously, the sub-BBM evolving inside the large clearing does not feel the effect of the traps. This is why the result is insensitive to the nature and the parameters of the trapping mechanism. The details of this discussion are presented in the proof of the lower bound of Theorem~\ref{thm1}. 
\end{remark}

\section{Preparations}\label{section3}

In this section, we collect some preparatory results that will later be used in the proof of Theorem~\ref{thm1}.

Let us introduce some further notation that will be used throughout the paper. We use $\mathbb{N}$ as the set of positive integers and $\mathbb{R}_+$ as the set of positive real numbers. For $x\in\mathbb{R}^d$, we denote by $|x|$ the Euclidean distance of $x$ to the origin. For a set $A\subseteq\mathbb{R}^d$ and $x\in\mathbb{R}^d$, we define their sum in the sense of sum of sets as $x+A:=\{x+y:y\in A\}$. For a set $A\subseteq\mathbb{R}^d$, we denote by $\partial A$ its boundary in $\mathbb{R}^d$. For two functions $f,g:\mathbb{R}_+\to\mathbb{R}_+$, we write $g(t)=o(f(t))$ and $f(t)\sim g(t)$ if $g(t)/f(t)\rightarrow 0$ and $g(t)/f(t)\rightarrow c$ for some positive constant $c>0$ as $t\rightarrow\infty$, respectively. For an event $A$, we use $A^c$ to denote its complement, and $\mathbbm{1}_A$ its indicator function. We will use $c$, $c_1$, $c_2$, etc. to denote generic constants, whose values may change from line to line. The notation $c(p)$ or $c_p$ will be used to mean that the constant $c$ depends on the parameter $p$.

\subsection{Tubular estimate}

Let $X=(X_t)_{t\geq 0}$ denote a standard Brownian motion in $d$ dimensions, and $(\mathbf{P}_x:x\in\mathbb{R}^d)$ be the laws of Brownian motion started at $x$ with corresponding expectations $(\mathbf{E}_x:x\in\mathbb{R}^d)$. We now state a previous result, which is taken from \cite{S1993}, and will be used in the proof of Lemma~\ref{lemma2}. It concerns a Brownian motion in a free environment, and gives a lower bound on the probability that a Brownian motion stays within a fixed distance from the central axis of a `tube' connecting its starting point to a given point. For $x,y\in\mathbb{R}^d$ and $t>0$, consider the line segment $\{x+(y-x)s/t:0\leq s\leq t\}$. One may refer to the set $\{z\in\mathbb{R}^d:\inf_{0\leq s\leq t}|z-(x+(y-x)s/t)|<a\}$ as a tube (or cylinder) of radius $a$ connecting $x$ and $y$ in $\mathbb{R}^d$.

\begin{propa}[Tubular estimate for Brownian motion; \cite{S1993}]
Let $x,y\in\mathbb{R}^d$ be fixed. There exists a constant $c_d>0$ that depends only on dimension $d$ such that for all $t>0$ and $b>0$,
\begin{equation} 
\mathbf{P}_x\left(\sup_{0\leq s\leq t}\bigg|X_s-\left(x+\frac{s}{t}(y-x)\right)\bigg|<b \right) \geq c_d\exp\left[-\frac{\lambda_d t}{b^2}-\frac{|y-x|^2}{2t} \right]. \nonumber
\end{equation}
\end{propa}

\subsection{Survival probability}

We first give a definition concerning special random subsets of $\mathbb{R}^d$ in the environment $\omega$, followed by a previous result, which gives an a.s.-environment in the soft obstacle setting. Then, we prove a preliminary result on the survival probability of the BBM among soft obstacles.
 
\begin{definition} \label{def1}
We call $A\subseteq\mathbb{R}^d$ a \emph{clearing} in the random environment $\omega$ if $A\subseteq K^c$. By a \emph{clearing of radius $r$}, we mean a ball of radius $r$ which is a clearing. 
\end{definition}

\begin{propb}[Large almost-sure clearings, soft obstacles; \cite{S1998}]
Let the random environment in $\mathbb{R}^d$ be given by \eqref{eqpotential} and \eqref{eqtraprule}. Then, on a set of full $\mathbb{P}$-measure, there exists $\ell_0>0$ such that for all $\ell>\ell_0$ the cube $[-\ell,\ell]^d$ contains a clearing of radius 
\begin{equation} \label{eq0}
R_\ell:=R_0(\log \ell)^{1/d}-(\log \log \ell)^2 ,\:\:\ell>1. 
\end{equation}
\end{propb}

\begin{proposition}[Survival probability for BBM among soft obstacles]\label{prop1}
Let the random environment in $\mathbb{R}^d$ be given by \eqref{eqpotential} and \eqref{eqtraprule}. Then, under the soft killing rule, on a set of full $\mathbb{P}$-measure, 
$$ 0<P^\omega(S)<1  . $$
\end{proposition}

\begin{proof}
Recall the definitions of $S_t$ and $S$ from \eqref{survival}. It is clear that $P^\omega(S_t)$ is nonincreasing in $t$, and bounded below by zero. Hence, $\lim_{t\to\infty} P^\omega(S_t) = P^\omega(S)$ exists. We will show that on a set of full $\mathbb{P}$-measure there exist constants $c_1=c_1(\omega)$ and $c_2=c_2(\omega)$ such that 
\begin{equation}
0<c_1\leq P^\omega(S_t)\leq c_2<1  \label{eqsurvival}
\end{equation}
for all large $t$. To prove the upper bound in \eqref{eqsurvival}, use the single-particle Brownian survival asymptotics among soft obstacles from \cite[Theorem 2.5]{S1993}, which implies that on a set of full $\mathbb{P}$-measure the survival probability goes to zero as $t\to\infty$. This means, there exists $t_0=t_0(\omega)$ such that for all $t\geq t_0$ the probability of survival up to time $t$ is at most $1/2$. This implies, since the branching and motion mechanisms in a BBM are independent, that on a set of full $\mathbb{P}$-measure, $P^\omega(S_t)\leq 1-\exp(-\beta t_0)/2$ for all $t\geq t_0$.  

To prove the lower bound in \eqref{eqsurvival}, define
\begin{equation}
\Omega_s=\{\omega\in\Omega:\exists\:\ell_1=\ell_1(\omega), \:\forall\:\ell\geq \ell_1, \:[-\ell,\ell]^d \:\text{contains a clearing of radius}\: R_\ell\} . \label{eqomegas}
\end{equation}
From Proposition B, we know that $\mathbb{P}(\Omega_s)=1$. On the other hand, by the proof of \cite[Thm.\ 5.5.4, p.193]{E2014}, there exists a critical radius, say $R_{cr}$, which is given by $\lambda_{d,R_{cr}}=\beta$, such that for any $R>R_{cr}$ the probability $p_R$ that at least one particle of BBM has not left $B(0,R)$ ever, is positive. Now let $\omega\in\Omega_s$, and choose $R=R(\omega)$ so that 
$$ R>R_{cr}+1 \quad \text{and} \quad e^{(2R/R_0)^d}>\ell_1(\omega)   , $$
where $\ell_1$ is as introduced in \eqref{eqomegas}. Then, in the environment $\omega$, by definition of $R_\ell$ and $\Omega_s$, the box 
$$ C(\omega,d):=\left[-e^{(2R/R_0)^d},e^{(2R/R_0)^d}\right]^d $$
contains a clearing of radius $R$. Let $B(x_0,R)$ be this clearing where $x_0=x_0(\omega)$ and $B(x_0,R)\subseteq C(\omega,d)$. Consider the following survival strategy for the BBM. Over $[0,1]$, avoid being killed by the trap field and send the initial particle to $B(x_0,1)$. We may (but don't have to) suppress the branching over $[0,1]$ so that the initial particle is still alive at time $1$. Let this joint strategy have probability $p_s$. Observe that $p_s>0$ since the box $C(\omega,d)$ is fixed and the killing function $W$ is bounded. Then, over $[1,\infty)$, we know from the proof of \cite[Thm.\ 5.5.4, p.193]{E2014} that since $R>R_{cr}+1$ at least one particle of the sub-BBM that is initiated at time $1$ from within $B(x_0,1)$ does not ever leave the clearing $B(x_0,R)$ with probability $p_R>0$. Hence, by the Markov property applied at time $1$, for all $t\geq 1$, 
\begin{equation} \nonumber
 P^\omega(S_t)\geq P^\omega(S) \geq p_s p_R>0 .
\end{equation}
This completes the proof of the lower bound in \eqref{eqsurvival}. 
\end{proof}

\subsection{Expected mass}

A first consideration for the mass of BBM among soft obstacles is to calculate its expectation and obtain a formula to the leading order which holds in almost every environment. The expected mass formula below will also be explicitly used in the proof of the upper bound of Theorem~\ref{thm1}.

\begin{proposition}[Expected mass for BBM among soft obstacles] \label{prop2}
On a set of full $\mathbb{P}$-measure,
\begin{equation*} 
E^\omega[N_t]=\exp\left[\beta t-c(d,\nu)\frac{t}{(\log t)^{2/d}}(1+o(1))\right].
\end{equation*}
\end{proposition}

\begin{proof}
Consider a general branching mechanism in $K$ such that when inside $K$ particles branch according to the offspring law $(p_k)_{k\geq 0}$ as opposed to binary branching. Let $\mu_1=\sum_{k=1}^\infty k p_k$ be the associated mean number of offspring. Observe that soft killing under the potential $V$ together with complete suppression of branching inside the obstacles is tantamount to the offspring law $(p_k)_{k\geq 0}$ with $p_0=1$ and branching rate $V(x,\omega)$ inside $K$. In this way, both the branching rate and the offspring mean depend on position as
\begin{align}
\beta(x,\omega) &=\beta\,\mathbbm{1}_{K^c(\omega)}(x)+V(x,\omega)\,\mathbbm{1}_{K(\omega)}(x), \label{branchinglaw} \\ 
\mu(x,\omega)   &=2\,\mathbbm{1}_{K^c(\omega)}(x) \label{branchinglaw2}  . 
\end{align} 
Note that $p_0=1$ implies $\mu_1=0$. By the construction in \eqref{eqtraprule}, $V=V\mathbbm{1}_K$. Define $m(x,\omega)=\mu(x,\omega)-1$ and $m_1=\mu_1-1$. Then, $\beta(x,\omega)m(x,\omega)=\beta-(\beta+V)\mathbbm{1}_K$. Applying the classical first moment formula for spatial branching processes $\omega$-wise (see for instance \cite[Lemma 1]{GHK2022} for a more general version), and using \eqref{branchinglaw} and \eqref{branchinglaw2}, we obtain
\begin{align}
E^\omega[N_t] &= \mathbf{E}_0 \left[\exp\left(\int_0^t \beta(X_s,\omega)m(X_s,\omega) ds \right)\right] \nonumber \\
&= e^{\beta t }\mathbf{E}_0 \left[\exp\left(-\int_0^t (\beta\mathbbm{1}_{K(\omega)}(X_s)+V(X_s,\omega))ds\right) \right]. \nonumber
\end{align}
The expectation on the right-hand side is the survival probability up to $t$ of a single Brownian motion among soft obstacles with killing function 
\begin{equation}
\widetilde{W}(x)=\beta\mathbbm{1}_{K_0}(x)+W(x), \label{eqw}
\end{equation}
where $K_0$ denotes the compact set on which $W$ is supported, except that the first term in \eqref{eqw} is not summed on the overlapping compacts. This, nonetheless, does not affect the asymptotic behavior of the survival probability (see \cite[Remark 4.2.2]{S1998}). Note that the function $\widetilde{W}$ is also positive, bounded, measurable, and compactly supported. Hence, the result follows from \cite[Theorem 4.5.1]{S1998}. 
\end{proof}

\subsection{Large-deviations for BBM in an expanding ball}

For a generic standard Brownian motion $X=(X_t)_{t\geq 0}$ and a Borel set $A\subseteq\mathbb{R}^d$, define $\sigma_A=\inf\{s\geq 0:X_s\notin A\}$ to be the first exit time of $X$ out of $A$. We now describe the model of \emph{BBM with deactivation at a boundary}, which was introduced in \cite{O2021}. For a Borel set $A\subseteq\mathbb{R}^d$, denote by $\partial A$ the boundary of $A$. Consider a family of Borel sets $B=(B_t)_{t\geq 0}$. Let $Z^B=(Z_t^{B_t})_{t\geq 0}$ be the BBM deactivated at $\partial B$, which can be obtained from $Z$ as follows: for each $t\geq 0$, start with $Z_t$, and delete from it any particle whose ancestral line up to $t$ has exited $B_t$ to obtain $Z^{B_t}_t$. This means, $Z^{B_t}_t$ consists of particles of $Z_t$ whose ancestral lines up to $t$ have been confined to $B_t$ over the time period $[0,t]$ (but may have left $B_s$ at an earlier time $s$). 

The following result is the first part (the low $\kappa$ regime) of \cite[Theorem 2]{O2021}, and will be used in the proof of the main result. It gives the large-time asymptotic behavior of the probability that the mass of BBM deactivated at the boundary of a subdiffusively expanding ball $B=(B_t)_{t\geq 0}$ is atypically small. 

\begin{thmc}[Lower large-deviations for mass of BBM in an expanding ball; Theorem 2, \cite{O2021}] \label{thma}
Let $r:\mathbb{R}_+ \to \mathbb{R}_+$ be increasing such that $r(t)\to\infty$ as $t\to\infty$ and $r(t)=o(\sqrt{t})$. Also, let $\gamma:\mathbb{R}_+ \to \mathbb{R}_+$ be defined by $\gamma(t)=e^{-\kappa r(t)}$, where $\kappa>0$ is a constant. For $t>0$, set $B_t=B(0,r(t))$, $p_t=\mathbf{P}_0(\sigma_{B_t}\geq t)$, and $n_t=|Z_t^{B_t}|$. Then, for any $0<\kappa\leq \sqrt{\beta/2}$, 
\begin{equation}
\underset{t\rightarrow\infty}{\lim}\,\frac{1}{r(t)}\log P\left(n_t < \gamma_t p_t e^{\beta t}\right)= -\kappa. \nonumber
\end{equation}
\end{thmc}

\section{The quenched environment}\label{section4}

In this section, we construct the a.s., that is, the quenched environment for the problem of BBM among soft obstacles. The following lemma is a stronger version of \cite[Lemma 1]{O2021} and \cite[Lemma 4.5.2]{S1998}, and will be used to prepare the a.s.-environment for the soft obstacle problem.

\begin{lemma}[A.s.\ clearings, soft obstacles, $d\geq 2$]\label{lemma1}
Let $a\in\mathbb{R}_+$, $b\geq 0$, and $c\geq 1$ be fixed, and define the function $f:\mathbb{R}_+\to\mathbb{N}$ by
$$ f(\ell)=\left\lceil e^{a\ell^{3/2}} \right\rceil .$$
For $\ell>0$, let $x_1,\ldots,x_{f(\ell)}$ be any set of $f(\ell)$ points in $\mathbb{R}^d$, and define the cubes $C_{j,\ell}=x_j+[-\ell,\ell]^d$, $1\leq j\leq f(\ell)$. Then, in $d\geq 2$, on a set of full $\mathbb{P}$-measure, there exists $\ell_0>0$ such that for each $\ell\geq \ell_0$, each of $C_{1,\ell},C_{2,\ell},\ldots,C_{f(\ell),\ell}$ contains a clearing of radius $R_\ell+b$, where $R_\ell$ is given by
\begin{equation}   
R_\ell:=\frac{R_0}{5^{1/d}}(\log c\ell)^{1/d} ,\:\:\: \ell>1. \label{eqrell}    
\end{equation}     
\end{lemma}

\begin{proof}
Let $x_1,x_2,\ldots$ be a sequence of points in $\mathbb{R}^d$, and $C_{j,\ell}:=x_j+[-\ell,\ell]^d$ for $j=1,2,\ldots$ For $k\geq 0$, let $A_{\ell,k}$ be the event that there is a clearing of radius $R_\ell+k$ in each $C_{1,\ell},C_{2,\ell},\ldots,C_{f(\ell+1),\ell}$. Also, for $k\geq 0$, define
$$E_{\ell,k}=\{[-\ell,\ell]^d\:\:\text{contains a clearing of radius $R_\ell+k$}\}.$$
Then, by the homogeneity of the PPP and the union bound, 
\begin{equation}
\mathbb{P}(A_{\ell,k}^c)\leq f(\ell+1) \mathbb{P}(E_{\ell,k}^c). \label{eq1lemma3}
\end{equation}
We now estimate $\mathbb{P}(E_{\ell,k}^c)$. Partition $[-\ell,\ell]^d$ into smaller cubes of side length $2(R_\ell+k)$. Inscribe a ball of radius $R_\ell+k$ in each smaller cube, and bound $\mathbb{P}(E_{\ell,k}^c)$ from above as 
\begin{equation}
\mathbb{P}(E_{\ell,k}^c)\leq \left[1-e^{-\nu\omega_d (R_\ell+k)^d}\right]^{\lfloor\ell/(R_\ell+k)\rfloor^d} 
\leq \exp\left[-\left\lfloor\frac{\ell}{R_\ell+k}\right\rfloor^d e^{-\nu\omega_d (R_\ell+k)^d}\right], \label{eq2lemma3}
\end{equation}
where the estimate $1+x\leq e^x$ is used.
Let 
$$\alpha_\ell:=\left\lfloor \ell/(R_\ell+k) \right\rfloor^d e^{-\nu\omega_d (R_\ell+k)^d}.$$
Then, using \eqref{eqro} and \eqref{eqrell}, and that $\log\lfloor\ell/(R_\ell+k) \rfloor\geq \log\frac{\ell}{2 R_\ell}$ for large $\ell$, it follows that
\begin{align}
\log \alpha_\ell&\geq d\log \ell-d\log(2 R_\ell)-\nu\omega_d(R_\ell+k)^d \nonumber \\
&\geq d\log \ell-d\log(2 R_\ell)-\frac{d}{R_0^d}\left[(19/18)^{1/d}R_\ell\right]^d \nonumber \\
&\geq \left(d-\frac{2d}{9}\right)\log \ell \geq \frac{14}{9}\log\ell, \label{eq3lemma3}
\end{align}
for all large $\ell$, where the last line follows due to \eqref{eqrell} and since $d\geq 2$ by assumption. It follows from \eqref{eq2lemma3} and \eqref{eq3lemma3} that for a given $k>0$, for all large $\ell$,
\begin{equation*}
\mathbb{P}(E_{\ell,k}^c)\leq e^{-\alpha_\ell}\leq e^{-\ell^{14/9}}.
\end{equation*}
Then, \eqref{eq1lemma3} yields
\begin{equation} \label{borelcantelli}
\sum_{n=1}^\infty \mathbb{P}\left(A_{n,k}^c\right) \leq c(n_0)+\sum_{n=n_0}^\infty \left\lceil e^{a(n+1)^{3/2}} \right\rceil e^{-n^{14/9}}<\infty, 
\end{equation} 
where $c(n_0)$ is a constant that depends on $n_0$. Applying Borel-Cantelli lemma, we conclude that with $\mathbb{P}$-probability one, only finitely many $A_{n,k}^c$ occur. That is, $\mathbb{P}(\Omega_s)=1$, where
\begin{equation} \label{eqlemma10}
\Omega_s=\{\omega:\exists n_1=n_1(\omega)\:\:\forall n\geq n_1,\:\:\text{each}\:\:C_{1,n},\ldots,C_{f(n+1),n}\:\:\text{has a clearing of radius $R_n+k$} \} .
\end{equation}
Let $\omega_0\in\Omega_s$, and $n_1=n_1(\omega_0)$ be as in \eqref{eqlemma10}. Observe that
\begin{equation} 
R_{n+1}-R_n\leq \frac{R_0}{5^{1/d}}\left[(\log c(n+1))^{1/d}-(\log c n)^{1/d}\right] \rightarrow 0, \quad n\to\infty . \nonumber
\end{equation}
In particular, there exists $n_2\in\mathbb{N}$ such that for all $n\geq n_2$, $R_{n+1}-R_n\leq 1$. Choose $k=b+1$. (So far the choice of $k>0$ was arbitrary.) Denote by $a\vee b$ the maximum of the numbers $a$ and $b$. To complete the proof, it suffices to show that in the environment $\omega_0$ for each $\ell\geq n_3:=n_1\vee n_2$, each $C_{1,\ell},\ldots,C_{f(\ell),\ell}$ contains a clearing of radius $R_\ell+b$. Take $\ell\geq n_3$ so that there exists $n\geq n_3$ with $n\leq \ell\leq n+1$. Fix this integer $n$. Then, since $R_\ell$ is increasing in $\ell$, we have
\begin{equation} \label{tavsancik2}
R_\ell+b \leq R_{n+1}+b \leq R_n+1+b = R_n+k .
\end{equation}
Furthermore, 
\begin{equation} \label{tavsancik30}
f(\ell)=\left\lceil e^{a\ell^{3/2}} \right\rceil \leq \left\lceil e^{a(n+1)^{3/2}} \right\rceil = f(n+1) .
\end{equation}
Then, \eqref{eqlemma10}, \eqref{tavsancik2} and \eqref{tavsancik30} imply that for $\ell\geq n_3$, each of $C_{1,\ell},\ldots,C_{f(\ell),\ell}$ contains a clearing of radius $R_\ell+b$. This completes the proof since the choice of $\omega_0\in\Omega_s$ was arbitrary and $\mathbb{P}(\Omega_s)=1$.
\end{proof}

Next, we use Lemma~\ref{lemma1} with a suitably chosen collection of points $\left(x_j:1\leq j\leq f(\ell)\right)$ and a set of parameters $\ell$, $a$, $c$ in order to prepare an a.s.-environment with `high' concentration of `large' clearings, that is, in which the covering radius of the `large' clearings is sufficiently small. We will use this a.s.-environment as the quenched setting for the problem of BBM among soft obstacles. 
 
\begin{proposition}[An a.s.-environment, soft obstacles, $d\geq 2$] \label{prop3}
Let $k>0$ be fixed, and $C(0,kt)=[-kt,kt]^d$ be the cube centered at the origin with side length $2kt$. Let $\rho:\mathbb{R}_+\to\mathbb{R}_+$ be such that
\begin{equation} \label{eqrho}
\rho(t)=(\log t)^{2/3}, \quad t>1  .
\end{equation}
For $b>0$, define the set of environments $\Omega_s=\Omega_s(k,b)$ as
\begin{equation} \label{eqenviron}
\Omega_s=\{\omega\in\Omega:\exists\:t_0\:\:\forall\:t\geq t_0,\:\:\forall\:x\in C(0,kt)\:\: \exists\:y\in B(x,\rho(t))\:\:\text{such that}\: B\left(y,R_{\rho(t)}+b\right)\subseteq K^c\}, 
\end{equation}
where $R_{\rho(t)}$ is as in \eqref{eqrell} with $c=1$. Then, in $d\geq 2$, $\mathbb{P}(\Omega_s)=1$.
\end{proposition}

\begin{proof}
Consider the simple cubic packing of $C(0,kt)$ with balls of radius $\rho(t)/(2\sqrt{d})$. Then, at most
\begin{equation} \label{eq12}
n_t:=\left\lceil \frac{kt}{\rho(t)/(2\sqrt{d})}  \right\rceil^d 
\end{equation}
balls are needed to completely pack $C(0,kt)$, say with centers $\left(z_j:1\leq j\leq n_t\right)$. For each $j$, let $B_t^j=B(z_j,\rho(t)/(2\sqrt{d}))$. Now consider generically a simple cubic packing of $\mathbb{R}^d$ by balls $(\mathcal{B}_j:j\in\mathbb{N})$ of radius $R>0$, and let $x\in\mathbb{R}^d$ be any point. It is easy to deduce from elementary geometry that $\min_j \max_{z\in \mathcal{B}_j}|x-z|<(\sqrt{d}/2)4R$, where $\sqrt{d}/2$ is the distance between the center and any vertex of the $d$-dimensional unit cube $C(0,1/2)$. Then, since the radius of the packing balls is $\rho(t)/(2\sqrt{d})$ in our case, it follows that
\begin{equation}  \label{eq1100}
\forall\,x\,\in C(0,kt), \quad
\underset{1\leq j\leq n_t}{\min}\:\underset{z\in B_t^j}{\max}\:|x-z|<\rho(t). 
\end{equation}
We now combine the simple cubic packing of $C(0,kt)$ and Lemma~\ref{lemma1}. Set $\ell=\rho(t)/(2\sqrt{d})=(\log t)^{2/3}/(2\sqrt{d})$ in Lemma~\ref{lemma1}. Then, $t=e^{(2\ell\sqrt{d})^{3/2}}$, and it follows from \eqref{eq12} that for all large $\ell$,
\begin{equation}
n_t\leq \frac{(2k)^d}{\ell^d} e^{d(2\ell\sqrt{d})^{3/2}} \leq e^{d(2\ell\sqrt{d})^{3/2}}. \nonumber
\end{equation}
Now, with the choices $\ell=\rho(t)/(2\sqrt{d})$, $a=d(2\sqrt{d})^{3/2}$, $c=2\sqrt{d}$ and $x_j=z_j$ for $j\leq n_t$, where $(z_j:1\leq j\leq n_t)$ are as above, in view of \eqref{eq1100} and since $\ell\to\infty$ as $t\to\infty$, Lemma~\ref{lemma1} implies the following. For fixed $k>0$ and $b>0$, on a set of full $\mathbb{P}$-measure, there exists $t_0>0$ such that for all $t\geq t_0$, $B(x,\rho(t))$ contains a clearing of radius $\frac{R_0}{5^{1/d}}(\log \rho(t))^{1/d}+b$ for each $x\in C(0,kt)$.
\end{proof}

In the subsequent proofs, $\Omega_s=\Omega_s(k,b)$ given in \eqref{eqenviron} with a suitable pair $(k,b)$, will be our quenched environment for the problem of BBM among soft obstacles.

The term `overwhelming probability' is henceforth used with a precise meaning, which is given as follows.
\begin{definition}[Overwhelming probability]
Let $(A_t)_{t>0}$ be a family of events indexed by time $t$, and $\mathcal{P}$ be the relevant probability. We say that $A_t$ occurs \emph{with overwhelming probability} if 
$$\underset{t\to\infty}{\lim}\mathcal{P}(A_t^c)=0 . $$ 
\end{definition}

\section{Hitting the moderate clearings}\label{section5}

In this section we show that on the set of full $\mathbb{P}$-measure developed in the previous section, that is, on $\Omega_s$ given in \eqref{eqenviron}, the BBM hits clearings of a certain size over $[0,t]$ for large $t$ with overwhelming $\widehat{P}^\omega$-probability. In the rest of the paper, two types of clearings will be considered according to size: \emph{moderate} clearings have $r(t)\sim (\log\log t)^{1/d}$ and \emph{large} clearings have $r(t)\sim (\log t)^{1/d}$, where $r=r(t)$ is used as the radius of a clearing. The following lemma will play a central role in the proof of the lower bound of Theorem~\ref{thm1}. It is on the hitting probability and the position of hitting over $[0,t]$ of a BBM among soft obstacles to moderate clearings conditioned on survival over $[0,t]$, and says that with overwhelming probability, the BBM hits such a clearing within the horizon $[-kt,kt]^d$ over $[0,t]$. As before we use $Z=(Z_t)_{t\geq 0}$ to denote a BBM in $d$ dimensions, $P_x$ as the law of a free BBM started with a single particle at position $x\in\mathbb{R}^d$, and $P_x^\omega$ as the conditional law of a BBM started with a single particle at position $x\in\mathbb{R}^d$ in the environment $\omega$. Set $P^\omega=P^\omega_0$. Also, recall the definition of $R_0$ from \eqref{eqro}. The \emph{range} (accumulated support) of $Z$ is the process defined by 
\begin{equation} 
\mathcal{R}(t)=\bigcup_{0\leq s\leq t} \text{supp}(Z_s). \nonumber
\end{equation}

\begin{lemma}[Hitting probability of BBM to moderate clearings]\label{lemma2}
Let $r:\mathbb{R}_+\to\mathbb{R}_+$ be such that 
\begin{equation} \label{part1eqradius}
r(t)=\frac{1}{3}\frac{R_0}{5^{1/d}}\left(\frac{2}{3}\right)^{1/d}(\log \log t)^{1/d}   , \quad t>e . 
\end{equation}
Let $k>\sqrt{2\beta}$ be fixed. For $\omega\in\Omega$ and $t>0$, define  
\begin{equation}
\Phi_t^\omega=\{x\in\mathbb{R}^d:B(x,r(t))\subseteq K^c(\omega)\}, \quad\:\: \widehat{\Phi}_t^\omega=\Phi_t^\omega \cap [-kt,kt]^d .  \label{eqphi}
\end{equation}
Then, in $d\geq 2$, there exists $\Omega_1\subseteq\Omega$ with $\mathbb{P}(\Omega_1)=1$ such that for every $\omega\in\Omega_1$, 
\begin{equation} 
\underset{t\to \infty}{\lim}\: P^\omega\left(\mathcal{R}(t)\cap \widehat{\Phi}_t^\omega=\emptyset \:\big\vert\: S_t\right)=0 . \nonumber
\end{equation}
\end{lemma}

\begin{proof}
Call $x\in\mathbb{R}^d$ a \emph{good point} for $\omega\in\Omega$ at time $t$ if $B(x,r(t))$ is a clearing (see Definition~\ref{def1}) in the random environment $\omega$. That is,
$$ \Phi_t^\omega:=\{x\in\mathbb{R}^d:B(x,r(t))\subseteq K^c(\omega)\} $$
is the set of good points associated to the pair $(\omega,t)$. Given $\omega\in\Omega$, for $t>0$ define the events
$$ E_t=E_t(\omega)=\{\mathcal{R}(t)\cap \widehat{\Phi}_t^\omega=\emptyset\} . $$
In words, $E_t$ is the event that the BBM does not hit a good point inside $[-kt,kt]^d$ associated to the pair $(\omega,t)$ over $[0,t]$. By Proposition~\ref{prop1}, on a set of full $\mathbb{P}$-measure, there exists $c_1=c_1(\omega)>0$ such that for all large $t$,
\begin{equation} \label{part1eq11}
P^\omega(E_t \mid S_t) = \frac{P^\omega(E_t \cap S_t)}{P^\omega(S_t)}\leq c_1 P^\omega(E_t\cap S_t) .
\end{equation}  
In the rest of the proof, we bound $P^\omega(E_t\cap S_t)$ from above in a typical environment $\omega$.

For $t>1$, introduce the time scale 
$$ h(t):=(\log t)^{2/3} . $$ 
For notational convenience\footnote{We would like to avoid the floor function in notation.}, suppose that $t/h(t)$ is an integer. Split the interval $[0,t]$ into $t/h(t)$ pieces as
$$ [0,h(t)],\: [h(t),2h(t)],\:\ldots\:,[t-h(t),t], $$
and for $j=1,2,\ldots,t/h(t)$, define the intervals $I_{j,t}$ as
$$  I_{j,t}=[(j-1)h(t),j h(t)] . $$
Next, for $t>e$, introduce two space scales as follows: use $\rho(t)$, previously defined in \eqref{eqrho}, as the larger space scale, and $r(t)$ given in \eqref{part1eqradius} as the smaller space scale. That is, we have
$$ \rho(t)=h(t)=(\log t)^{2/3}, \quad r(t)=\frac{1}{3}\frac{R_0}{5^{1/d}}\left(\frac{2}{3}\right)^{1/d}(\log \log t)^{1/d} .$$ 
Observe that $2r(t)\leq R_{\rho(t)}$ for all large $t$, where $R_\ell$ is as in \eqref{eqrell} with $c=1$ therein. Hence, for each $\omega\in\Omega_s(k,0)$ (see \eqref{eqenviron} for the definition), for all large $t$ any ball of radius $\rho(t)$ centered within $C(0,kt)$ contains a clearing of radius $2r(t)$. In the rest of the proof, we set $\Omega_s=\Omega_s(k,0)$ with $k>\sqrt{2\beta}$ fixed. Recall that $\mathbb{P}(\Omega_s)=1$ by Proposition~\ref{prop3}. 

For an interval $I\subseteq [0,\infty)$, define the range of $Z$ over $I$ as 
\begin{equation} 
\mathcal{R}(I)=\bigcup_{s\in I} \text{supp}(Z_s). \nonumber
\end{equation} 
Next, for $t>1$ and $j=1,2,\ldots,t/h(t)$, define the events
$$  E_{j,t}=\{\mathcal{R}(I_{j,t})\cap \widehat{\Phi}_t^\omega=\emptyset\}, \:\: \quad S_{j,t}:=\{N_{jh(t)}\geq 1\}.   $$
Observe that 
\begin{equation} \label{eq1part3}
E_t \cap S_t = \bigcap_{j=1}^{t/h(t)} (E_{j,t}\cap S_{j,t}).
\end{equation}
Also, for $t>0$, let
\begin{equation} \label{eqmt}
M_t:=\inf\{r\geq 0:\mathcal{R}(t)\subseteq B(0,r)\} 
\end{equation} 
be the radius of the minimal ball containing the range of BBM at time $t$, and for $t>1$ and $j=1,2,\ldots,t/h(t)$ define the events
\begin{equation} \label{eqfj}
F_{j,t}=\{M_{j h(t)}\leq k j h(t)\}, \:\: \quad F_t=F_{1,t} \cap \ldots \cap F_{t/h(t),t}  .
\end{equation}
We now apply repeated conditioning at times $h(t), 2h(t), \ldots, t-h(t)$, and at each intermediate time $j h(t)$, we will throw away the rare event $F_{j,t}^c$. Note that $F_{j,t}^c$ is indeed a rare event since $k>\sqrt{2\beta}$ by assumption and it is well-known that the \emph{speed} of a strictly dyadic BBM is $\sqrt{2\beta}$. Let $\omega\in\Omega_s$. Then, using \eqref{eq1part3}, \eqref{eqfj}, and the union bound, 
\begin{align} 
P^\omega(E_t, S_t)&\leq P^\omega(E_t, S_t, F_t) + P^\omega(F_{1,t}^c) + \ldots + P^\omega(F_{t/h(t),t}^c) \nonumber \\
& = P^\omega\left(\bigcap_{j=1}^{t/h(t)} (E_{j,t}, S_{j,t}, F_{j,t}) \right)+\sum_{j=1}^{t/h(t)}P^\omega(F_{j,t}^c)  \nonumber \\
&= P^\omega\left(\bigcap_{j=2}^{t/h(t)} (E_{j,t}, S_{j,t}, F_{j,t}) \:\bigg\vert\: E_{1,t}, S_{1,t}, F_{1,t}\right)P^\omega(E_{1,t}, S_{1,t}, F_{1,t})+ \sum_{j=1}^{t/h(t)}P^\omega(F_{j,t}^c) \nonumber .
\end{align}
Iterating the argument above at times $2h(t), \ldots, t-h(t)$, and noting that $S_{j,t}=\cap_{k=1}^j S_{k,t}$, we obtain
\begin{equation} \nonumber
P^\omega(E_t, S_t)\leq P^\omega\left(E_{1,t}, S_{1,t}, F_{1,t}\right) \prod_{j=2}^{t/h(t)} P^\omega\left(E_{j,t}, S_{j,t}, F_{j,t} \:\bigg\vert\: S_{j-1,t}\, ,\:\bigcap_{k=1}^{j-1} (E_{k,t}, F_{k,t})\right) + \sum_{j=1}^{t/h(t)}P^\omega(F_{j,t}^c)  
\end{equation}
from which it follows that
\begin{equation} \label{eqmain2}
P^\omega(E_t, S_t)\leq P^\omega(E_{1,t})\prod_{j=2}^{t/h(t)}P^\omega\left(E_{j,t} \:\bigg\vert\: S_{j-1,t}\, ,\:\bigcap_{k=1}^{j-1} (E_{k,t}, F_{k,t})\right)  + \sum_{j=1}^{t/h(t)}P^\omega(F_{j,t}^c).
\end{equation}
In the rest of the proof, we find an upper bound that is valid for large $t$ on the right-hand side of \eqref{eqmain2} in an environment $\omega\in\Omega_s$.

\medskip

\textbf{(i) Upper bound on $P^\omega(F_{j,t}^c)$ in any environment}

To estimate $P^\omega(F_{j,t}^c)=P^\omega(M_{jh(t)}>k j h(t))$, we need some control on the spatial spread of the BBM at time $j h(t)$. We start by noting an $\omega$-wise comparison between a BBM among soft obstacles and a free BBM. (Recall that \emph{free} refers to the model where $V\equiv 0$, that is, there is no killing and the BBM branches at rate $\beta$ everywhere in $\mathbb{R}^d$.) The following stochastic domination is clear since the presence of $V>0$ can only kill particles as well as suppressing their branching, and otherwise has no effect on the motion of particles. As before, we use $(P_y:y\in\mathbb{R}^d)$ for the laws of a free BBM starting with a single particle at $y\in\mathbb{R}^d$, and set $P=P_0$.
\begin{remark}[Comparison $1$, free environment versus soft killing]
For $t>0$ and $B\subseteq\mathbb{R}^d$, let $Z_t(B)$ denote the mass of $Z$ that fall inside $B$ at time $t$. Then, for all $y\in\mathbb{R}^d$, $B\subseteq \mathbb{R}^d$ Borel, $k\in\mathbb{N}$, and $t>0$,
\begin{equation} \label{eqcomp1}
P_{y}(Z_t(B)<k) \leq P_{y}^\omega(Z_t(B)<k) \quad \text{for each $\omega\in\Omega$} .
\end{equation}
\end{remark}
Then, for any $r>0$, it follows by taking $B=(B(0,r))^c$ and $k=1$ in \eqref{eqcomp1} that 
\begin{equation} \label{eqcomp2}
P(M_t\leq r) \leq P^\omega(M_t\leq r) ,
\end{equation}
where $M_t$ is as defined in \eqref{eqmt}. Observe that $M_t/t$ is a kind of speed for the BBM, and measures the spread of $Z$ from the origin over the time interval $[0,t]$.

Let $\mathcal{N}_t$ denote the set of particles of $Z$ that are alive at time $t$, and for $u\in\mathcal{N}_t$, let $(Y_u(s))_{0\leq s\leq t}$ denote the ancestral line up to $t$ of particle $u$. By the \emph{ancestral line up to $t$} of a particle present at time $t$, we mean the continuous trajectory traversed up to $t$ by the particle, concatenated with the trajectories of all its ancestors. Note that when $V\equiv 0$, $(Y_u(s))_{0\leq s\leq t}$ is identically distributed as a Brownian trajectory $(X_s)_{0\leq s\leq t}$ for each $u\in\mathcal{N}_t$. Also note that $N_t=|\mathcal{N}_t|$. Then, using the union bound, for $\gamma>0$,
\begin{equation} \label{eq1part1}
P\left(M_t>\gamma t\right)= P\left(\exists\, u\:\in \mathcal{N}_t:\sup_{0\le s\le t}|Y_u(s)|>\gamma t\right) \le E[N_t]\:\mathbf{P}_0\left(\sup_{0\le s\le t}|X_s|>\gamma t\right). 
\end{equation}
It is a standard result that $E[N_t]=\exp(\beta t)$ (one can deduce this, for example, from \cite[Sect.\ 8.11]{KT1975}), and from \cite[Lemma 5]{OCE2017} we have that $\mathbf{P}_0\left(\sup_{0\le s\le t}|X_s|>\gamma t\right)=\exp[-\gamma^2 t/2(1+o(1))]$. Set $\gamma=k$ and replace $t$ by $j h(t)$ in \eqref{eq1part1}. Then, combining \eqref{eqcomp2} and \eqref{eq1part1}, and recalling that $k>\sqrt{2\beta}$, 
\begin{align} \label{eq2part1}
P^\omega(F_{j,t}^c)=P^\omega(M_{j h(t)}>k j h(t))&\leq P(M_{j h(t)}>k j h(t)) \nonumber \\
&\leq E[N_{jh(t)}]\:\mathbf{P}_0\bigg(\sup_{0\le s\le j h(t)}|X_s|>kj h(t)\bigg) \nonumber \\
&= \exp[jh(t)(\beta-k^2/2)(1+o(1))] .
\end{align}
It follows from \eqref{eq2part1} that when $k>\sqrt{2\beta}$,
\begin{equation} \label{maineqpiece2}
\sum_{j=1}^{t/h(t)}P^\omega(F_{j,t}^c) \leq t \exp[-h(t)(k^2/2-\beta)(1+o(1))] .
\end{equation}

\bigskip

\textbf{(ii) Upper bound on $P^\omega\left(E_{1,t}\right)$ in a typical environment}

\medskip

Next, for $\omega\in\Omega_s$, we find an upper bound on $P^\omega\left(E_{1,t}\right)$ that is valid for large $t$. We will estimate $P^\omega(E_{1,t}^c)$ from below, and in order to do that, since $E_{1,t}^c=\{\mathcal{R}(I_{1,t})\cap \widehat{\Phi}^\omega_t\neq \emptyset\}$, we look for a hitting strategy to $\widehat{\Phi}^\omega_t$. We start by noting an $\omega$-wise comparison between a Brownian motion (BM) among soft obstacles and a BBM among soft obstacles. The following comparison is obvious since each particle of BBM follows a Brownian trajectory while alive under the laws $P^\omega_y$. 
\begin{remark}[Comparison $2$, BM versus BBM]
Let $(\mathbf{P}^\omega_y:y\in\mathbb{R}^d)$ be the laws under which $X$ is a BM starting at position $y$ and is killed at rate $V=V(x,\omega)$ in the environment $\omega$. Then, for all $y\in\mathbb{R}^d$, $t>0$ and $B\subseteq \mathbb{R}^d$ Borel,
\begin{equation} \nonumber
\mathbf{P}_y^\omega\left((\cup_{0\leq s\leq t}\{X_s\}) \cap B\neq\emptyset\right) \leq P_{y}^\omega\left(\mathcal{R}(t)\cap B\neq\emptyset\right) \quad \text{for each $\omega\in\Omega$} .
\end{equation}
That is, even under the killing potential $V$, it is easier for a BBM to hit any set $B$ than it is for a single BM. We note that a similar comparison between a free BBM and a free BM also holds. 
\end{remark}
The remark above implies that
\begin{equation} \label{eqprelim}
\mathbf{P}_0^\omega\left((\cup_{0\leq s\leq h(t)}\{X_s\}) \cap \widehat{\Phi}^\omega_t \neq\emptyset\right) \leq P^\omega(E_{1,t}^c) ,
\end{equation}
and hence, it suffices to estimate $\mathbf{P}_0^\omega\left((\cup_{0\leq s\leq h(t)}\{X_s\}) \cap \widehat{\Phi}^\omega_t \neq\emptyset\right)$ from below. Let $\omega\in\Omega_s$ and choose $t$ large enough. Then, in the environment $\omega$, $B_{1,t}:=B(0,\rho(t))$ contains a clearing of radius $2r(t)$, hence a ball of radius $r(t)$, say $\mathcal{B}_{1,t}$, that is entirely contained in $\Phi^\omega_t$. Since $\rho(t)\leq kt$ for large $t$, we have $\mathcal{B}_{1,t}\subseteq \widehat{\Phi}^\omega_t\cap B_{1,t}$. Let $\mathbf{e}$ be the unit vector in the direction of the center of $\mathcal{B}_{1,t}$ in $\mathbb{R}^d$. Consider the following strategy for a standard BM: over $[0,h(t)]$, avoid being killed by the potential $V$, stay in the tube 
$$  T_t:=\left\{z\in\mathbb{R}^d:\inf_{0\leq s\leq h(t)}\bigg|z-\rho(t)\mathbf{e}\,\frac{s}{h(t)}\bigg|<r(t) \right\}   , $$
and hit $\mathcal{B}_{1,t}$. The probability of this joint event is at least 
\begin{equation} \label{eqjoint}
\exp\left[-h(t)\sup_{x\,\in\,T_t} V(x,\omega)\right] \mathbf{P}_0\left(\sup_{0\leq s\leq h(t)} \bigg|X_s-\rho(t)\mathbf{e}\,\frac{s}{h(t)}\bigg|<r(t)\right) ,
\end{equation}
where the second factor is a lower bound for the probability that the particle stays inside $T_t$ and hits $\mathcal{B}_{1,t}$. Indeed, if the event $\left\{\sup_{0\leq s\leq h(t)} \big|X_s-\rho(t)\mathbf{e}\,\frac{s}{h(t)}\big|<r(t)\right\}$ is realized, this means the particle is in $B(\rho(t)\mathbf{e},r(t))$ at time $h(t)$, which, by continuity of Brownian paths, implies that it must have hit $\mathcal{B}_{1,t}$ over the interval $[0,h(t)]$. By \cite[Lemma 4.5.2]{S1998}, 
\begin{equation} \label{eqjoint1}
\exp\left[-h(t)\sup_{x\,\in\,T_t} V(x,\omega)\right] \geq \exp[-h(t)\log \rho(t)].
\end{equation}
By the tubular estimate in Proposition A, and since $r(t)\geq 1$ for all large $t$,  
\begin{equation} \label{eqjoint2}
\mathbf{P}_0\left(\sup_{0\leq s\leq h(t)} \bigg|X_s-\rho(t)\mathbf{e}\,\frac{s}{h(t)}\bigg|<r(t)\right)\geq  c_d \exp\left[-\lambda_d h(t)-\frac{\rho^2(t)}{2h(t)} \right]   
\end{equation}
for all large $t$, where $c_d>0$ is a constant that only depends on the dimension. Then, it follows from \eqref{eqprelim}-\eqref{eqjoint2} that for all large $t$, 
\begin{equation} \label{eqkey}
P^\omega(E_{1,t}^c) \geq c_d \exp\left[-\left(h(t)\log \rho(t)+\lambda_d h(t)+\frac{\rho^2(t)}{2h(t)}\right) \right] .
\end{equation}
Using $\rho(t)=h(t)$, we see that $\exp[-2h(t)\log h(t)]$ is smaller than the right-hand side of \eqref{eqkey} for large $t$, and hence conclude that for all large $t$,
\begin{equation} \label{eqkey2}
P^\omega(E_{1,t}) \leq 1-e^{-2h(t)\log h(t)} . 
\end{equation}
This completes the estimate for $P^\omega\left(E_{1,t}\right)$ when $\omega\in\Omega_s$.


\bigskip

\textbf{(iii) Applying the Markov property at times $h(t), 2h(t),\ldots,t-h(t)$.}

\medskip

For $t>1$ and $j=2,\ldots,t/h(t)$, abbreviate 
$$ A_{j-1,t}:= S_{j-1,t}\cap (\cap_{i=1}^{j-1}(E_{i,t}, F_{i,t})) .$$
We now estimate $P^\omega(E_{j,t} \mid A_{j-1,t})$ in \eqref{eqmain2}. Recall the definition of $F_{j,t}$ from \eqref{eqfj}. Observe that conditional on $A_{j-1,t}$, at time $(j-1)h(t)$ the BBM has at least one particle alive and all particles are within a distance of $k(j-1)h(t)$ from the origin. Pick any particle that is alive\footnote{For concreteness, we may for instance pick the one that is closest to the origin at time $(j-1)h(t)$.} at time $(j-1)h(t)$, call it $u^*$, and let $y^{(j)}_t:=Y_{u^*}((j-1)h(t))$ denote its position at that time. Since $F_{j-1,t}\subset A_{j-1,t}$ and $k(j-1)h(t)\leq k(t-h(t))$ for all $j=1,\ldots,t/h(t)$, conditional on $A_{j-1,t}$ we have that $y^{(j)}_t\leq k(t-h(t))$. Now let $\omega\in\Omega_s$ and choose $t$ large enough. Then, in the environment $\omega$, $B_{j,t}:=B\big(y^{(j)}_t,\rho(t)\big)$ contains a clearing of radius $2r(t)$, hence a ball of radius $r(t)$, say $\mathcal{B}_{j,t}$, that is entirely contained in $\Phi^\omega_t$ and also in $[-kt,kt]^d$ since $y^{(j)}_t\leq k(t-h(t))$ and $\rho(t)=h(t)$. That is, $\mathcal{B}_{j,t}\subseteq \widehat{\Phi}^\omega_t\cap B_{j,t}$. Then, applying the Markov property at time $(j-1)h(t)$, a tubular estimate argument similar to the one used in step (ii) for the case $j=1$ yields that for all large $t$,
\begin{equation} \label{eqmain3}
P^\omega\left(E_{j,t}\mid A_{j-1,t}\right)\leq 1-e^{-2h(t)\log h(t)} ,\quad j=2,\ldots,t/h(t)  .
\end{equation}
Then, combining \eqref{eqmain2}, \eqref{maineqpiece2}, \eqref{eqkey2} and \eqref{eqmain3} yields the following conclusion. Provided $k>\sqrt{2\beta}$, in any environment $\omega\in\Omega_s$ for all large $t$,
\begin{align}
P^\omega(E_t\cap S_t)&\leq \left[1-e^{-2h(t)\log h(t)}\right]^{t/h(t)}+ t \exp[-h(t)(k^2/2-\beta)(1+o(1))] \nonumber \\
& \leq \exp\left[-e^{-2h(t)\log h(t)}\frac{t}{h(t)}\right]+ t \exp[-h(t)(k^2/2-\beta)(1+o(1))] \nonumber 
\end{align}
where we have used the estimate $1+x\leq e^x$. Since $h(t)=(\log t)^{2/3}$, it follows that    
\begin{equation} \nonumber
\underset{t\to\infty}{\lim}\: P^\omega(E_t\cap S_t) = 0 .
\end{equation} 
This completes the proof of Lemma~\ref{lemma2} in view of \eqref{part1eq11}. 
\end{proof}

\begin{remark}
Note that any conditioning on the events $S_t$ (or on $S$) changes the law of the BBM. In particular, the ancestral lines are no longer Brownian. In the proof of Lemma~\ref{lemma2}, the conditioning on $S_t$ was carried out in stages over successive subintervals $[(j-1)h(t),jh(t)]$ in order to work with Brownian paths.

Also, observe that the trivial bound $P^\omega(E_t\cap S_t)\leq P^\omega(E_t)$ is not useful for proving Lemma~\ref{lemma2}, because the event $E_t$ is realized if the entire process is killed before hitting $\widehat{\Phi}_t^\omega$, which has a probability bounded below by a positive number uniformly for all large $t$. 

In case of mild obstacles, where there is no killing but only a suppression of branching inside traps, each ancestral line is Brownian under $P^\omega$, and therefore it is sufficient to prove the counterpart of Lemma~\ref{lemma2} for a single Brownian motion (see Lemma 2 in \cite{O2021}). In contrast, in case of soft obstacles, one has to estimate $P^\omega(E_t\cap S_t)$ for the entire BBM. On the event $E_t\cap S_t$ there is at least one particle at time $t$ whose ancestral line is Brownian under $P^\omega$ and who has survived up to time $t$, but we don't know which particle, and a standard union bound argument over all possible particles existing at time $t$ would not be successful in showing that $P^\omega(E_t\cap S_t)\to 0$ as $t\to\infty$.

We emphasize that all of the aforementioned difficulties arise due to the soft killing mechanism in the model, which was not present in the mild obstacle problem. 
\end{remark}

\section{Proof of Theorem~\ref{thm1}}\label{section6}

\subsection{Proof of the upper bound}

For the proof of the upper bound, let $\Omega_1\subset\Omega$ be the intersection of the sets of full $\mathbb{P}$-measure in Proposition~\ref{prop1} and Proposition~\ref{prop2}, and let $\omega\in\Omega_1$. Recall that the law $\widehat{P}^\omega$ is defined by $\widehat{P}^\omega(\:\cdot\:)=P^\omega(\:\cdot\: \mid S)$, and denote by $\widehat{E}^\omega$ the corresponding expectation. Write
$$  \widehat{E}^\omega[N_t] P^\omega(S) = E^\omega[N_t \mathbbm{1}_{S}] \leq E^\omega[N_t] . $$
We know from Proposition~\ref{prop1} that $0<P^\omega(S)<1$, and therefore,
\begin{equation} \label{eqsqueeze}
\widehat{E}^\omega[N_t] \leq \frac{E^\omega[N_t]}{P^\omega(S)}.
\end{equation}
By the Markov inequality, we then have
\begin{equation}
\widehat{P}^\omega\left(N_t>\exp\left[\beta t+\frac{(-c(d,\nu)+\varepsilon)t}{(\log t)^{2/d}}\right]\right)\leq \widehat{E}^\omega[N_t]\exp\left[-\beta t+\frac{(c(d,\nu)-\varepsilon)t}{(\log t)^{2/d}}\right], \nonumber
\end{equation}
which, along with \eqref{eqsqueeze} and Proposition~\ref{prop2} implies that
$$ \widehat{P}^\omega\left((\log t)^{2/d}\left(\frac{\log N_t}{t}-\beta\right)+c(d,\nu)>\varepsilon\right)\leq \exp\left[-\varepsilon t(\log t)^{-2/d}+o\left(t(\log t)^{-2/d}\right)\right] .$$
This proves the upper bound of the LLN in \eqref{eqthm1}.


\subsection{Proof of the lower bound}

The proof of the lower bound is split into three parts for better readability. Let $\varepsilon>0$. In what follows, in a typical environment $\omega$, we find an upper bound that is valid for large $t$ on
\begin{equation} 
\widehat{P}^\omega\left((\log t)^{2/d}\left(\frac{\log N_t}{t}-\beta\right)+c(d,\nu)<-\varepsilon\right)=\widehat{P}^\omega\left(N_t<\exp\left[t\left(\beta-\frac{c(d,\nu)+\varepsilon}{(\log t)^{2/d}}\right)\right]\right)  .  \nonumber
\end{equation}
Throughout the proof, we assume that $d\geq 2$ so that Lemma~\ref{lemma2} is applicable.

\bigskip

\textbf{\underline{Part 1}: Upper bound on exponentially few total mass}

For a Borel set $B\subseteq \mathbb{R}^d$ and $t\geq 0$, as before $Z_t(B)$ denotes the mass of $Z$ that fall inside $B$ at time $t$. In this part of the proof, we will show that on a set of full $\mathbb{P}$-measure, for any $0<\delta<\beta$, the event 
\begin{equation} \label{part1event}
A_t=\{\exists\:z_0=z_0(\omega)\in[-kt,kt]^d \:\:\text{such that}\:\: Z_t(B(z_0,r(t)))\geq e^{\delta t} \} ,
\end{equation}
with $r(t)$ as in \eqref{part1eqradius} and $k>\sqrt{2\beta}$, occurs with overwhelming $\widehat{P}^\omega$-probability. Observe that $A_t$ corresponds to producing exponentially many particles and keeping them close to each other and also to the origin at time $t$. The main ingredient in this part of the proof will be Lemma~\ref{lemma2}.

Let $0<\delta<\beta$, and choose $\alpha$ such that $0<\alpha<1-\delta/\beta$. Also, for concreteness, set $k=\sqrt{3\beta}$. Recall the definitions of $\Phi_t^\omega$ (the set of good points associated to the pair $(\omega,t)$) and $\widehat{\Phi}_t^\omega=\Phi_t^\omega \cap [-kt,kt]^d$ from \eqref{eqphi}. Split the interval $[0,t]$ into two pieces as $[0,\alpha t]$ and $[\alpha t,t]$. We will show that with overwhelming probability, a particle of the BBM hits a point in $\widehat{\Phi}_{\alpha t}^\omega$, say $z_0\in [-k\alpha t,k\alpha t]^d$, over $[0,\alpha t]$, and then the sub-BBM emanating from this particle starting at $z_0$ produces at least $e^{\delta t}$ particles over $[\alpha t,t]$ inside $B(z_0,r(\alpha t))$. 

For $t>0$, define the events
$$ E_t:=\{ \mathcal{R}(\alpha t)\cap \widehat{\Phi}_{\alpha t}^\omega \neq \emptyset \} .$$
Let $\tau=\tau(\omega)=\inf\{s>0:\mathcal{R}(s)\cap \widehat{\Phi}_{\alpha t}^\omega\neq\emptyset\}$ be the first time that $Z$ hits a good point within the cube $[-k\alpha t,k\alpha t]^d$ associated to the pair $(\omega,\alpha t)$. Observe that $E_t=\{\tau\leq \alpha t \}$. Estimate
\begin{align} 
P^\omega\left(A_t^c \cap S_{\alpha t}\right) &= P^\omega\left(A_t^c \cap S_{\alpha t} \cap E_t^c\right) + P^\omega\left(A_t^c \cap S_{\alpha t} \cap E_t\right) \nonumber \\
&\leq P^\omega\left(E_t^c \mid S_{\alpha t} \right) + P^\omega\left(A_t^c \mid E_t\right) . \label{eqnewnew}
\end{align}
By Lemma~\ref{lemma2}, on a set of full $\mathbb{P}$-measure, 
\begin{equation} \label{part1eq6}
\underset{t\to\infty}{\lim}P^\omega(\mathcal{R}(\alpha t)\cap \widehat{\Phi}_{\alpha t}^\omega=\emptyset \mid S_{\alpha t})=\underset{t\to\infty}{\lim}P^\omega(E_t^c \mid S_{\alpha t})=0.
\end{equation}
Conditional on $E_t=\{\tau\leq\alpha t\}$, let $z_0$ be the (random) point where $Z$ first hits $\widehat{\Phi}_{\alpha t}^\omega$. Now apply the strong Markov property of BBM at time $\tau$, and then apply Theorem C to the sub-BBM initiated at time $\tau$ from position $z_0$ by the particle that first hits $\widehat{\Phi}_{\alpha t}^\omega$. Note that $t-\tau\geq (1-\alpha)t$, and by definition of $\widehat{\Phi}_{\alpha t}^\omega$, $B(z_0,r(\alpha t))$ is a clearing with $z_0\in[-k\alpha t,k\alpha t]^d$. In detail, for $t>1$ let
$$ s:=(1-\alpha)t,\quad \hat{r}(s)=\frac{1}{3}\frac{R_0}{5^{1/d}}\left(\frac{2}{3}\right)^{1/d}\left[\log\log\left(\frac{\alpha s}{1-\alpha}\right)\right]^{1/d}, \quad B_s:=B(0,\hat{r}(s)).$$ 
Observe that $\hat{r}(s)=r(\alpha t)$. Also, $\delta t<(1-\alpha)\beta t=\beta s$ due to the choice $\alpha<1-\delta/\beta$. Next, let 
$$ p_s:=\mathbf{P}_{z_0}(\sigma_{B(z_0,\hat{r}(s))}\geq s) = \mathbf{P}_0\left(\sigma_{B_s}\geq s\right) , $$
where, as before, $\mathbf{P}_x$ denotes the law of a standard BM started at $x$, and $\sigma_A=\inf\{s\geq 0:X_s\notin A\}$ denotes the first exit time of the BM out of $A$. By a standard result on Brownian confinement in balls (see for instance Proposition B in \cite{O2021}),
\begin{equation} \nonumber
p_s=\exp\left[-\frac{\lambda_d s}{(\hat{r}(s))^2}(1+o(1))\right]. 
\end{equation}
Then, on a set of full $\mathbb{P}$-measure, Theorem C upon setting $\gamma_s=\exp[-\sqrt{\beta/2}\,\hat{r}(s)]$ for instance implies that for all large $t$,
\begin{equation} \label{part1eq7}
P^\omega\left(A_t^c \mid E_t \right) \leq P\left(\big | Z_s^{B_s} \big |< e^{\delta t}\right) \leq
P\left(\big | Z_s^{B_s} \big |< e^{-\sqrt{\beta/2}\,\hat{r}(s)} p_s e^{\beta s}\right) = e^{-\sqrt{\beta/2}\,\hat{r}(s)(1+o(1))},
\end{equation} 
where $A_t$ is as in \eqref{part1event}, $Z^B=(Z_u^{B_u})_{u\geq 0}$ is a BBM with deactivation at $\partial B$ as in Theorem C, and we have used in the first inequality that $B(z_0,r(\alpha t))$ is a clearing . Finally, in view of $s=(1-\alpha)t$, we reach the following conclusion via \eqref{eqnewnew}, \eqref{part1eq6} and \eqref{part1eq7}. On a set of full $\mathbb{P}$-measure, for all large $t$,
\begin{align} \label{part1eq8}
P^\omega\left(A_t^c \mid S \right)&=\frac{P^\omega(A_t^c\cap S)}{P^\omega(S)}\leq c(\omega) P^\omega(A_t^c\cap S) \nonumber \\
&\leq c(\omega) P^\omega(A_t^c\cap S_{\alpha t}) \to 0,\:\:\: t\to\infty.
\end{align}
where we have used Proposition~\ref{prop1} in the first inequality, and that $S\subseteq S_{r}$ for any $r>0$ in the second inequality. This completes the first part of the proof of the lower bound of Theorem~\ref{thm1}. 


\bigskip

\textbf{\underline{Part 2}: Time scales within $[0,t]$ and moving a particle into a large clearing}

Introduce two time scales, $m(t)$ and $\ell(t)$, where $m(t)=o(t)$ and $\ell(t)\log \ell(t)=o(m(t))$. We will split the time interval $[0,t]$ into three pieces: $[0,m(t)]$, $[m(t),m(t)+\ell(t)]$ and $[m(t)+\ell(t),t]$. (This way of splitting $[0,t]$ is different from the corresponding splitting of $[0,t]$ used in \cite{E2008} and \cite{O2021} for the proofs of the mild obstacle problem.) More precisely, let $\ell,m:\mathbb{R}_+\to\mathbb{R}_+$ be such that
\begin{enumerate}
	\item[(i)] $\lim_{t\rightarrow\infty}\ell(t)=\infty$,
	\item[(ii)] $\lim_{t\rightarrow\infty} \frac{\log t}{\log \ell(t)}=1$,
	\item[(iii)] $\ell(t)\log \ell(t)=o(m(t))$,
	\item[(iv)] $m(t)=o(\ell^2(t))$,
	\item[(v)] $m(t)=o(t(\log t)^{-2/d})$.
\end{enumerate} 
In this part of the proof, our goal is to show that on a set of full $\mathbb{P}$-measure with overwhelming $\widehat{P}^\omega$-probability, at time $m(t)+\ell(t)$ there is a particle within distance $1$ of the center, say $x_0$, of a large clearing. This can be achieved by combining two partial strategies as follows. Firstly, sufficiently many particles are produced over $[0,m(t)]$ and kept close to each other at time $m(t)$, and then at least one of the sub-BBMs initiated by these particles at time $m(t)$ contributes a particle to $B(x_0,1)$ at time $m(t)+\ell(t)$. 

\medskip

\noindent \textbf{Partial strategy 1.} Let $0<\delta<\beta$ and $I(t)=\lfloor e^{\delta m(t)} \rfloor$. Then, since $\lim_{t\to\infty}m(t)=\infty$, it follows from \eqref{part1event} and \eqref{part1eq8} that on a set of full $\mathbb{P}$-measure,
\begin{equation} \label{part2eq1}
\underset{t\to\infty}{\lim}P^\omega(A_{m(t)}^c \mid S) = 0,
\end{equation}   
where
\begin{equation}
A_{m(t)}=\left\{\exists\,z_0=z_0(\omega)\in[-k m(t),k m(t)]^d \:\:\text{with}\:\: Z_{m(t)}\left(B(z_0,c[\log\log m(t)]^{1/d})\right)\geq I(t)\right\}  \label{eqamt}
\end{equation}
with $0<c<R_0$ (see \eqref{part1eqradius}). We now choose our `new origin' as the point $z_0$ in \eqref{eqamt} for the rest of the proof, that is, for the evolution of the system over $[m(t),t]$. 




\medskip

\noindent \textbf{Partial strategy 2.} Let $z_0=z_0(\omega)\in[-k m(t),k m(t)]^d$ be as in \eqref{eqamt}. By a similar argument as in the proof of Proposition~\ref{prop3}, it follows from a close-packing of $[-k m(t),k m(t)]^d$ by balls of radius $\ell(t)/(2\sqrt{d})$ together with \cite[Lemma 1]{O2021} (which is an extension of Proposition B) upon choosing $n=d+1$, $a=1$, and $\ell=\ell(t)/(2\sqrt{d})$ therein that on a set of full $\mathbb{P}$-measure, for all large $t$ any ball of radius $\ell(t)$ centered within $[-k m(t),k m(t)]^d$ contains a clearing of radius  
\begin{equation}
R(t)+1=R_{\ell(t)}+1 \asymp R_0[\log(\ell(t))]^{1/d} \asymp R_0[\log t]^{1/d},  \label{eqbiggyradius}
\end{equation}
where $R_{\ell(t)}$ is as in \eqref{eq0}, and we have used assumption (ii). In \eqref{eqbiggyradius} and hereafter, we use $f(t)\asymp g(t)$ to mean $f(t)/g(t)\to 1$ as $t\to\infty$. To see why the choice $n=d+1$ in \cite[Lemma 1]{O2021} works, observe that the number of balls of radius $\ell(t)/(2\sqrt{d})$ needed to completely pack $[-k m(t),k m(t)]^d$ is at most
$$ \left\lceil \frac{k m(t)}{\ell(t)/(2\sqrt{d})} \right\rceil^d \leq  \left(\frac{k\ell^2}{\ell}\right)^d \leq \ell^{d+1} $$
for all large $t$, where we have used assumption (iv) in the first inequality. In particular, on a set of full $\mathbb{P}$-measure $B(z_0(\omega),\ell(t))$ contains a clearing of radius $R(t)+1$ for all large $t$. Let $x_0=x_0(\omega)$ denote the center of this clearing. We will next show that on a set of full $\mathbb{P}$-measure the event
$$ C_t:=\{\exists\,x_0=x_0(\omega)\in\mathbb{R}^d \:\:\text{with}\:\: Z_{m(t)+\ell(t)}(B(x_0,1))>0 \:\:\text{and}\:\: B(x_0,R(t)+1)\subseteq K^c\}  $$
occurs with overwhelming $P^\omega$-probability conditional on $A_{m(t)}$. 



Consider a particle inside $B(z_0,c[\log\log m(t)]^{1/d})$ at time $m(t)$, and call it generically particle $u$. Let $q_u(t)$ be the probability that the sub-BBM initiated by $u$ at time $m(t)$ contributes a particle to $B(x_0,1)$ at time $m(t)+\ell(t)$. For an upper bound on $q_u(t)$, we consider the worst case scenario: 
\begin{enumerate}
	\item[(a)] assume that $u$ is located at the boundary of $B(z_0,c[\log\log m(t)]^{1/d})$ at time $m(t)$ in the opposite direction of $x_0$ with respect to the `origin' $z_0$,
	\item[(b)] neglect possible branching of $u$ over $[m(t),m(t)+\ell(t)]$,
	\item[(c)] assume that the Brownian path initiated by $u$ travels through the trap field $K$ over the entire interval $[m(t),m(t)+\ell(t)]$. 
\end{enumerate}	
By \cite[Lemma 4.5.2]{S1998}, on a set of full $\mathbb{P}$-measure, we have  
\begin{equation} 
\sup_{[z_0-\ell(t),z_0+\ell(t)]^d} V(\:\cdot\:,\omega)= o(\log\ell(t)), \quad t\to\infty.  \nonumber
\end{equation}
Then, for each particle $u$ that is inside $B(z_0,c[\log\log m(t)]^{1/d})$ at time $m(t)$, in view of the inequalities $c[\log\log m(t)]^{1/d}\leq \ell(t)$ and $|x_0-z_0|\leq \ell(t)$, we have for all large $t$,
\begin{equation} \label{eq7}
q_u(t)\geq \exp\left[-\frac{(2\ell(t))^2}{2\ell(t)}(1+o(1))-\ell(t)\log\ell(t)\right] \geq \exp\left[-2\ell(t)\log\ell(t)\right] =: p(t),
\end{equation}
where the first term in the exponent on the right-hand side comes from a linear Brownian displacement and the second term from surviving the killing over $[m(t),m(t)+\ell(t)]$. Then, by the Markov property and the independence of particles present at time $m(t)$, we have
\begin{equation} \label{eq8}
P^\omega(C_t^c \mid A_{m(t)}) \leq (1-p(t))^{I(t)} = e^{-p(t) I(t)},  
\end{equation}
where we have used that $1+x\leq e^x$. Note that in order to keep the probability of the unwanted event $C_t^c$ small, we aimed at a high enough $p(t) I(t)$ throughout the argument. Finally, by \eqref{eq7} and \eqref{eq8}, we reach the following conclusion. On a set of full $\mathbb{P}$-measure, for all large $t$,
\begin{equation} \label{eq9}
P^\omega(C_t^c \mid A_{m(t)}) \leq \exp\left[-e^{-2\ell(t)\log\ell(t)+\delta m(t)}\right] .
\end{equation}
Since $\delta>0$ and due to the assumptions (ii) and (iii), the right-hand side of \eqref{eq9} is superexponentially small in $t$.

\bigskip

\textbf{\underline{Part 3}: BBM in the large clearing}

This part of the proof is similar to the corresponding part as for the mild obstacle case, because over the remaining time interval $[m(t)+\ell(t),t]$ the BBM grows freely inside the clearing $B(x_0,R(t)+1)$, and hence is insensitive to the nature of the traps. For $t>0$, define the events
$$ D_t:=\left\{\exists\,x_0=x_0(\omega)\in\mathbb{R}^d \:\:\text{with}\:\: Z_t\left(B(x_0,R(t)+1)\right)\geq \exp\left[t\left(\beta-\frac{c(d,\nu)+\varepsilon}{(\log t)^{2/d}}\right)\right]\right\}.  $$
We argue as follows. Let $\Omega_0$ be the set of environments for which $0<P^\omega(S)<1$. We know from Proposition~\ref{prop1} that $\mathbb{P}(\Omega_0)=1$. For each $\omega\in\Omega_0$, set $c(\omega)=1/P^\omega(S)$, and estimate
\begin{align} 
P^\omega(D_t^c \mid S) &= \frac{1}{P^\omega(S)}\left[P^\omega(D_t^c\cap S\cap C_t)+P^\omega(D_t^c\cap S\cap C_t^c)\right] \nonumber\\
&\leq c(\omega) \left[P^\omega(D_t^c\cap S_t\cap C_t)+P^\omega(D_t^c\cap S\cap C_t^c)\right] \nonumber \\
&\leq c(\omega) \left[P^\omega(D_t^c \mid C_t)+P^\omega(C_t^c \mid S)\right], \label{eq4}
\end{align}
where we have used that $S\subseteq S_t$ in the first inequality. The second term on the right-hand side of \eqref{eq4} can be estimated as follows for $\omega\in\Omega_0$:
\begin{align} 
P^\omega(C_t^c \mid S) &= \frac{1}{P^\omega(S)}\left[P^\omega\left(C_t^c\cap S\cap A_{m(t)}\right)+P^\omega\left(C_t^c\cap S\cap A_{m(t)}^c\right)\right] \nonumber \\
&\leq c(\omega)\left[P^\omega\left(C_t^c \mid A_{m(t)}\right)+P^\omega\left(A_{m(t)}^c \mid S\right)\right] \nonumber.
\end{align}
It then follows from \eqref{part2eq1} and \eqref{eq9} that on a set of full $\mathbb{P}$-measure, 
\begin{equation}
\lim_{t\to\infty}P^\omega(C_t^c \mid S)=0 \label{eqfok}.
\end{equation}
Next, we turn our attention to $P^\omega(D_t^c\mid C_t)$ in \eqref{eq4} and show that on a set of full $\mathbb{P}$-measure,
\begin{equation} 
\underset{t\to\infty}{\lim} P^\omega(D_t^c \mid C_t) = 0 . \label{eq11}
\end{equation} 
Conditional on the event $C_t$, let $v$ be the name of the particle that is closest to $x_0$ at time $m(t)+\ell(t)$ and $y_0$ denote its position at time $m(t)+\ell(t)$. Note that $|y_0-x_0|\leq 1$ on the event $C_t$. We will show via Proposition C that sufficiently many particles are produced inside $B(y_0,R(t))$ over the remaining interval $[m(t)+\ell(t),t]$. Let $\widehat{Z}$ be the sub-BBM initiated by particle $v$ at time $m(t)+\ell(t)$ starting from position $y_0$. Define $\widehat{R}:\mathbb{R}_+\to\mathbb{R}_+$ such that $\widehat{R}(t-(m(t)+\ell(t)))=R(t)$ for all large $t$. (Respecting the conditions (i)-(v), we may and do choose $m(t)$ and $\ell(t)$ such that $t-(m(t)+\ell(t))$ is increasing on $t\geq t_0$ for some $t_0>0$. Therefore, $t_1-(m(t_1)+\ell(t_1))=t_2-(m(t_2)+\ell(t_2))$ implies that $t_1=t_2$ for $t_1\wedge t_2\geq t_0$.) Next, let $s:=t-(m(t)+\ell(t))$, $\widehat{B}_s:=B(y_0,R(t))$ and
$$ p_s:=\mathbf{P}_{y_0}(\sigma_{\widehat{B}_s}\geq s) = \mathbf{P}_0\left(\sigma_{B(0,\widehat{R}(s))}\geq s\right) . $$ 
By the Markov property of $Z$ applied at time $m(t)+\ell(t)$, $\widehat{Z}$ is a BBM started with a single particle at $y_0$. Observe that $\widehat{B}_s$ is a clearing since $\widehat{B}_s\subseteq B(x_0,R(t)+1)$. Then, Theorem C upon setting $\gamma_s=\exp[-\sqrt{\beta/2}\widehat{R}(s)]$ implies that 
\begin{align}  
P^\omega\left(|\widehat{Z}_s|< e^{-\sqrt{\beta/2}\,\widehat{R}(s)} p_s e^{\beta s} \:\big|\: C_t \right)&\leq P_{y_0}\left(\big| Z_s^{\widehat{B}_s}\big|< e^{-\sqrt{\beta/2}\,\widehat{R}(s)} p_s e^{\beta s}\right) \nonumber \\
&=\exp\left[-\sqrt{\beta/2}\,\widehat{R}(s)(1+o(1))\right]. \label{eq309}
\end{align}
By \eqref{eqbiggyradius} and a standard result on Brownian confinement in balls (see for instance Proposition B in \cite{O2021}), and since $\widehat{R}(s)=R(t)$ and $c(d,\nu)=\lambda_d/R_0^2$,  
\begin{equation} \label{eq310}
p_s=\exp\left[-\frac{\lambda_d s}{\widehat{R}^2(s)}(1+o(1))\right]=\exp\left[-\frac{c(d,\nu)(t-(m(t)+\ell(t)))}{(\log \ell(t))^{2/d}}(1+o(1))\right]. 
\end{equation}
It follows from the assumptions (ii), (iii) and (v) that
$$ \frac{t-(m(t)+\ell(t))}{(\log \ell(t))^{2/d}} \asymp \frac{t}{(\log t)^{2/d}} , $$
by which we can continue \eqref{eq310} with
\begin{equation}
p_s=\exp\left[-\frac{c(d,\nu)t}{(\log t)^{2/d}}(1+o(1))\right].   \label{eq311}
\end{equation}
On the other hand, using that $s=t-(m(t)+\ell(t))$, we have for any $\varepsilon>0$,
$$  \exp\left[t\left(\beta-\frac{c(d,\nu)+\varepsilon}{(\log t)^{2/d}}\right)\right]=\exp\left[\beta s+\beta(m(t)+\ell(t))-\frac{(c(d,\nu)+\varepsilon)t}{(\log t)^{2/d}}\right]\leq e^{-\sqrt{\beta/2}\,\widehat{R}(s)} p_s e^{\beta s}  $$
for all large $t$, where we have used \eqref{eq311}, assumption (v), and that $\widehat{R}(s)=R(t)=o(t(\log t)^{-2/d})$ in passing to the inequality. Then, it follows from \eqref{eq309} and the definitions of $D_t$ and $\widehat{Z}$ that for all large $t$, 
$$ P^\omega(D_t^c \mid C_t) \leq P^\omega\left(|\widehat{Z}_s|< e^{-\sqrt{\beta/2}\,\widehat{R}(s)} p_s e^{\beta s} \:\big|\: C_t \right) \leq \exp\left[-\sqrt{\beta/2}\,\widehat{R}(s)(1+o(1))\right]   . $$
This proves, \eqref{eq11} holds on a set of full $\mathbb{P}$-measure, which, together with \eqref{eq4} and \eqref{eqfok}, imply that 
\begin{equation} \nonumber
P^\omega\left(N_t<\exp\left[t\left(\beta-\frac{c(d,\nu)+\varepsilon}{(\log t)^{2/d}}\right)\right] \:\bigg\vert\: S \right)\leq P^\omega(D_t^c \mid S)\to 0, \quad t\to\infty .
\end{equation} 
This completes the proof of the lower bound of Theorem~\ref{thm1}. We emphasize that over $[m(t)+\ell(t),t]$, the sub-BBM starting from $y_0$ at time $m(t)+\ell(t)$ and deactivated at $\partial B(y_0,R(t))$ doesn't feel the effect of traps; so for this part of the proof it doesn't matter whether traps have a killing mechanism or not.

\section{Further problems}\label{section8}

We conclude by discussing several further problems related to our model.

\bigskip

\noindent\textbf{Problem 1: The case $d=1$}. 

\medskip

We emphasize that the LLN for the case $d=1$ remains open in the soft obstacle setting. Here, we briefly explain why the current method fails when $d=1$.

We start by considering Lemma~\ref{lemma2}. Observe that the key estimate in the proof of Lemma~\ref{lemma2} is \eqref{eqkey}, where the right-hand side gives the cost of a hitting strategy to a moderate clearing, and involves a tubular estimate and an estimate on survival from killing over the interval $[0,h(t)]$. Note that in $d=1$, we do not need a full tubular estimate, but we still need a single Brownian motion to travel a distance of $\sim\rho(t)$ over a time interval of length $h(t)$. Even if we ignore the `tubular estimate' contribution in \eqref{eqkey}, we still have the factor $\exp[-h(t)\log\rho(t)]$, which is solely due to the survival from soft killing. Then, to show that $[P^\omega(E_{1,t})]^{t/h(t)}$ tends to zero as $t\to\infty$, at the very least we need
$$ \left[1-c_d e^{-h(t)\log\rho(t)}\right]^{t/h(t)}\leq \exp\left[-c_d e^{-h(t)\log\rho(t)}\frac{t}{h(t)}\right] \to 0 ,$$
which is true only if
\begin{equation} \label{eqneed}
t e^{-h(t)\log\rho(t)} \to \infty.
\end{equation}
An elementary inspection shows that the choices $h(t)=k_1\log t$ and $\rho(t)=k_2\log t$ will not satisfy \eqref{eqneed} no matter how small $k_1$ and $k_2$ are, and therefore for \eqref{eqneed} to hold one needs to choose $h(t)=o(\log t)=\rho(t)$ as $t\to\infty$.

We now turn our attention to the preparation of the a.s.-environment, which is based on securing a clearing of radius $\sim r(t)$ within each ball of radius $\sim\rho(t)$, where the union of the balls covers the box $[-kt,kt]^d$ (see Proposition~\ref{prop3}). For this argument to hold, we need such an $r(t)$-clearing inside each of 
\begin{equation} 
 \sim t/\rho(t)   \nonumber
\end{equation}
many balls. As we set $\ell=\rho(t)/2$, let us take a close look at Lemma~\ref{lemma1}. Observe that Lemma~\ref{lemma1} becomes stronger and more difficult to prove as each of $R_\ell$ and $f(\ell)$ are chosen larger. We now argue that the current proof of Lemma~\ref{lemma1} breaks down in $d=1$ in view of $\ell=\rho(t)/2$ under the requirement that $\rho(t)=o(\log t)$. Let us simply set $R_\ell=R$ for some constant $R>0$ to make the proof easier. Even in this case, the estimate \eqref{eq3lemma3} when $d=1$ yields
$$  \log \alpha_\ell \geq \log\ell-\log(2R)-R/R_0 .$$
Even if we ignore the middle term, this leads to
$$ e^{-\alpha_\ell}\leq \exp\left[-e^{(\log\ell-R/R_0)}\right]\leq \exp\left[-\ell e^{-R/R_0}\right]=\left(e^{-\ell}\right)^{e^{-R/R_0}}. $$
Then, for the Borel-Cantelli argument based on \eqref{borelcantelli} to hold, we can have $f(\ell)$ growing at most exponentially in $\ell$ (see \eqref{eq1lemma3} and \eqref{borelcantelli}). On the other hand, it is easy to see that when $\rho(t)=o(\log t)$, setting $\ell=\rho(t)/2$ followed by $f(\ell)\sim t^{d}/\rho(t)^d$ requires $f(\ell)$ to be superexponentially large in $\ell$.  

We may summarize our findings as follows. The current method fails when $d=1$, because when $d=1$ the estimate $[P^\omega(E_{1,t})]^{t/h(t)}$ in Lemma~\ref{lemma2} is incompatible with the Borel-Cantelli argument in the proof of Lemma~\ref{lemma1}, which is used to prepare the a.s.-environment for the soft obstacle problem. When $d=1$, there is no pair of choices for the time scale $h(t)$ and the space scale $\rho(t)$ under which both Lemma~\ref{lemma1} and Lemma~\ref{lemma2} hold upon setting $\ell=\rho(t)/2$.

\bigskip

\noindent\textbf{Problem 2: SLLN.}

\medskip

It would be desirable to improve the LLN in Theorem~\ref{thm1} to the corresponding SLLN if possible. To achieve this, one must control the probabilities of the `unwanted' events in the proof of the lower bound so that a Borel-Cantelli argument could be carried out to obtain the lower bound of the desired SLLN. Recall that Lemma~\ref{lemma2} was the key component in the proof of the lower bound of Theorem~\ref{thm1}. Therefore, a first and important step would be to improve Lemma~\ref{lemma2} to give a lower bound on the rate of decay to zero of $P^\omega\left(\mathcal{R}(t)\cap \widehat{\Phi}_t^\omega=\emptyset \:\big\vert\: S_t\right)$ as $t\to\infty$.

\bigskip

\noindent\textbf{Problem 3: Dominant region.}

\medskip

The current work, as well as \cite{E2008} and \cite{O2021}, study the total mass of the BBM among random obstacles, but don't make any claims on the geometric distribution of particles for large times although the proofs suggest that for large times `most' of the particles are to be found in `large' clearings which exist in a.e.-environment. 

Hence, a further problem concerns the geometric distribution of particles at large times. One natural question is, conditional on ultimate survival of the BBM, is there a \emph{dominant region} $B=B(\omega,t)$ in $\mathbb{R}^d$ such that an overwhelming proportion of particles are found inside $B$ for all large $t$. Recall that we write $Z_t(B)$ to denote the mass of $Z$ that fall inside $B$ at time $t$. More precisely, given an environment $\omega$ does there exist a region $B(\omega,t)\subset \mathbb{R}^d$ such that 
$$ \frac{Z_t(B(\omega,t)^c)}{N_t} \to 0, \quad t\to\infty $$
in some sense of convergence with respect to the law $\widehat{P}^\omega(\:\cdot\:)=P^\omega(\:\cdot\:\mid S)$ for a.e.\ $\omega$? If yes, this would mean that on a set of full $\mathbb{P}$-measure the mass accumulates in some $\omega$-dependent special subset of $\mathbb{R}^d$ for large times. This special subset, for instance, could be of a similar form as $\Phi_t^\omega$ in \eqref{eqphi} with a suitable radius function $r(t)$. We note that a similar problem in the setting of mild obstacles was listed as a further problem in \cite{E2008}. 

\bigskip

\noindent\textbf{Problem 4: Lower large-deviations.}

\medskip

In the course of the proof of the lower bound of Theorem~\ref{thm1}, we show that the rare event of atypically small mass for the BBM has probability decaying to zero as $t\to\infty$, but we do not find the rate of decay to zero of this probability. It is natural to look for this rate of decay and hence to obtain a precise lower-tail asymptotics for the mass of the BBM. That is, can we obtain an asymptotic result as $t\to\infty$ in the form
$$ P^\omega\left(N_t<\exp{\left[t\left(\beta-\frac{c(d,\nu)+\varepsilon}{(\log t)^{2/d}}\right)\right]}\right) = e^{-g(t)(1+o(1))}$$
that is valid in a.e.-environment, where the function $g=g_\varepsilon:\mathbb{R}_+\to\mathbb{R}_+$ with $\lim_{t\to\infty}g(t)=\infty$ is precisely identified?  

\bigskip

\noindent\textbf{Problem 5: Hard obstacles.}

\medskip

Let $\Pi$ be a Poisson point process in $\mathbb{R}^d$ with constant intensity $\nu>0$ as before, and consider the Poissonian trap field
\begin{equation}
K=K(\omega)=\bigcup_{x_i\in\:\text{supp}(\Pi)} \bar{B}(x_i,a)  \nonumber
\end{equation}
as in \eqref{eqtrapfield}, where the trap radius $a>0$ is fixed. The \emph{hard killing} rule for BBM is that each particle branches at the normal rate $\beta$ when outside $K$, and is immediately killed upon hitting $K$. This model may also be viewed as a BBM with individual killing at the boundary of the random set $K$. One can show, similar to the case of soft obstacles, that on a set of full $\mathbb{P}$-measure the entire BBM is killed with positive $P^\omega$-probability, and for meaningful results concerning the total mass one works under the law $\widehat{P}^\omega(\:\cdot\:)=P^\omega(\:\cdot\:\mid S)$.

It is observed from \cite[Theorem 1]{E2008} and \cite[Theorem 1]{O2021}, and then from the current work that the LLN for the total mass of BBM among random obstacles is quite robust to the details of the mass-reducing mechanism coming from the trap field, whether traps simply reduce the branching rate, or completely suppress the branching, or even apply soft killing to the particles. Therefore, it is reasonable to expect a similar LLN to hold even in the hard obstacle setting. 

It is known from the theory of site percolation that there exists $a_0>0$ such that when the trap radius $a$ satisfies $a\leq a_0$, a unique infinite trap-free component exists in a.e.-environment (see \cite{AKN1987} and \cite{S1993}). Here, a trap-free component refers to a connected region in $\mathbb{R}^d$ in which there is no atom of $\omega$. Let us denote by $\mathcal{C}$ this unique $\omega$-dependent infinite trap-free component, and call $A\subseteq\mathbb{R}^d$ \emph{accessible} if $A\subseteq\mathcal{C}$. Then, denoting the origin by $\mathbf{0}$, the conditions under which we may expect a LLN to hold (see \cite{S1993}) are as follows:
$$ \mathbb{P}-\text{a.s. on the set} \:\: \{\mathbf{0}\in\mathcal{C}\} \:\:  \text{and when} \:\: a\leq a_0 .$$

Note that in all cases of random Poissonian traps, the upper bound of the LLN is obtained by a first moment argument, and is unaffected by the nature of the traps since the first moment formula for the mass of BBM remains the same to the leading order due to the robustness of the single-particle Brownian survival asymptotics among the traps (see \cite{S1993}). In contrast, it is clear that the proof of the lower bound of the LLN becomes more difficult as the severity of the trapping mechanism increases in the following order: mild traps with a lower but positive rate of branching, mild traps with zero branching, traps with soft killing, and finally hard traps. 

In case of hard traps, the main extra challenge is due to the fact that although the concentration of moderate clearings close enough to the origin (within $[-kt,kt]^d$ for suitable $k>0$) is still high enough and there is at least one large clearing close enough to the origin just as in the case of soft obstacles, there is no guarantee that these clearings will be accessible to the BBM as they could be outside the unique infinite trap-free component. Therefore, all a.s.-clearings established in the proofs should further be qualified as accessible clearings.

\section*{Acknowledgements} The author would like to thank Dr.\ J\'anos Engl\"ander for fruitful discussions over the course of the writing of this manuscript.

\bibliographystyle{plain}

\end{document}